\documentclass[preprint,review,10pt]{amsart}
\usepackage{amsmath,amssymb,amsfonts,xspace}
\newtheorem{theorem}{Theorem}[section]

\theoremstyle{definition}
\newtheorem{definition}[theorem]{Definition}
\newtheorem{example}[theorem]{Example}

\newtheorem{proposition}[theorem]{Proposition}
\newtheorem{corollary}[theorem]{Corollary}

\theoremstyle{remark}
\newtheorem{remark}[theorem]{Remark}

\numberwithin{equation}{section}

\begin{document}

\title{$\phi$-$\delta$-$S$-primary hyperideals }

%    Information for first author
\author{Mahdi Anbarloei}
%    Address of record for the research reported here
\address{Department of Mathematics, Faculty of Sciences,
Imam Khomeini International University, Qazvin, Iran.
}
%    Current address

\email{m.anbarloei@sci.ikiu.ac.ir}
%    \thanks will become a 1st page footnote.

%    Information for second author
%\author{}
%\address{}
%\email{}
%\thanks{Support information for the second author.}

%    General info

%\date{September  , 2013.}

%\dedicatory{This paper is dedicated to our advisors.}
%------------------------------------------------------------------------------
%%%%%%%%%%%%%%%%%%%%%%%%%%%%%%%%%%%%%%%%%%%%%%%%%%%%%%%%%%%%%%%%%%%%%%%%%%%%%%%%%%%%%%%%%%%%%%%%%%%%%%

%%%%%%%%%%%%%%%%%%%%%%%%%%%%%%%%%%%%%%%%%%%%%%%%%%%%%%%%%%%%%%%%%%%%%%%%%%%%%%%%%%%%%%%%%%%%%%%%%%%%%%%%%%%%%%%%%%%%%%%%%%%%%%%%%%%%%%%%%
\begin{abstract}
Among many generalizations of primary hyperideals, weakly $n$-ary primary hyperideals and  $n$-ary $S$-primary hyperideals have been studied recently. Let $S$ be an $n$-ary multiplicative set of a commutative Krasner $(m,n)$-hyperring $K$ and,  $\phi$ and $\delta$ be reduction and expansion functions of hyperideals of $K$, respectively.  The purpose of this paper is to introduce $n$-ary $\phi$-$\delta$-$S$-primary hyperideals which serve as an extension of $S$-primary hyperideals with the help of  $\phi$ and $\delta$. We present some main results and examples explaining the sructure of this concept. We examine the relations of $n$-ary $S$-primary hyperideals with other classes of hyperideals and give some ways to connect them.
Moreover, we give some characterizations of this notion on direct product of commutative Krasner $(m, n)$-hyperrings.  
\end{abstract}
%%%%%%%%%%%%%%%%%%%%%%%%%%%%%%%%%%%%%%%%%%%%%%%%%%%%%%%%%%%%%%%%%%%%%%%%%%%%%%%%%
\maketitle

\section{Introduction}

Algebraic hyperstructures are a suitable  extension of classical algebraic
structures, allowing for the study of multi-valued operations and their consequences in a variety of mathematical fields. It was  Marty, who firstly introduced hyperstructure theory in 1934, during a presentation at the 8th Scandinavian Mathematical Congress \cite{s1}. Many researchers have investigated various hyperstructures \cite {amer2,  s3, davvaz1, davvaz2, s4, s10, jian}. Basic illustrations and results of $n$-ary hyperstructures can be seen in  \cite{s9, l2, l3,  rev1}.  The notion of Krasner $(m,n)$-hyperrings as an intermediate class of  the $(m,n)$-hyperrings and Krasner hyperrings  was introduced and studied  in \cite{d1}.  This algebraic structure involves an $m$-ary hyperoperation  and an $n$-ary  operation, satisfying certain axioms that generalize those of traditional rings. Some studies  focused on exploring the properties of Krasner $(m,n)$-hyperrings are found in   \cite{ sorc1,asadi,ma,nour,cons}. 
Primary hyperideals plays  a key role in the study of commutative hyperrings, providing a framework for comprehending the structure and properties of hyperideals in relation to prime hyperideals and other algebraic constructs.  The $n$-ary prime and $n$-ary primary hyperideals were studied in context of Krasner $(m,n)$-hyperrings in \cite{sorc1} and \cite{rev2}. Then, in \cite{mah3}, the idea of unifying this concepts in one frame was proposed.  The concep of $n$-ary $S$-primary hyperideals extending  the notion of $n$-ary primary hyperideals  by incorporating multiplicative subsets,  was defined and studied  in \cite{mah6}. Davvaz et al. investigated weakly $(k,n)$-absorbing  primary hyperideals as a generalization of  primary hyperideals in \cite{davvazz}. 
The notion of weakly $S$-primary hyperideals was discussed  in \cite{mah8}. The study often builds on previous work regarding weakly $n$-ary primary hyperideals and other related notions.

 The above mentioned algebraic structures motivated and inspired us to construct a broader extension of $S$-primary hyperideals with the help of reduction and expansion functions. In this paper,  we present the idea of $n$-ary $\phi$-$\delta$-$S$-primary hyperideals in a commutative Krasner $(m,n)$-hyperring $K$ where $\phi$  are $\delta$ are reduction and expansion functions, respectively, and $S$ is a multiplicative subset of $K$. Among many
results in this paper, we prove that if $P$ is an $n$-ary $\phi$-$\delta$-$S$-primary hyperideal of $K$ associated to $s \in S$ such that $rad(\phi(P)) \subseteq \phi(rad(P))$ and $rad(\delta(P)) \subseteq \delta(rad(P))$, then $rad(P)$ is an $n$-ary $\phi$-$\delta$-$S$-primary hyperideal of $K$ associated to $s \in S$, in Theorem \ref{1}. Example \ref{ignor}  shows that the conditions $``rad(\phi(P)) \subseteq \phi(rad(P))"$ and $``rad(\delta(P)) \subseteq \delta(rad(P))"$ in this theorem can not be ignored. In Theorem \ref{8}, we conclude that if $P$ is  a strongly  $n$-ary $\phi$-$\delta$-$S$-primary hyperideal of $K$ associated to $s \in S$ but is not an $n$-ary $\delta$-$S$-primary hyperideal, then $g(P^{(n)}) \subseteq \phi(P)$. However,  Example \ref{madar} indicates that a proper hyperideal $P$ of $K$ may not be $n$-ary $\phi$-$\delta$-$S$-primary when it holds $g(P^{(n)}) \subseteq \phi(P)$. In Theorem \ref{15}, under mentioned condictions, we obtain that  a hyperideal $P$ is an $n$-ary $\phi$-$\delta$-$S$-primary hyperideal of $K$ associated to $s \in S$ if and only if $S^{-1}P$ is an $n$-ary $\phi_S$-$\delta_S$-primary hyperideal of $S^{-1}K$ and $S^{-1}P \cap K =(P : s)$.
Finally, some characterizations of this concept are given on direct product of Krasner $(m,n)$-hyperrings.

%%%%%%%%%%%%%%%%%%%%%%
\section{Some basic definitions}
In what follow, we recall some  basic terms and definitions  that need in our work. 
Let $K$ be a non-empty set and   $P^*(K)=\{H \ \vert H \subseteq K\}$. A mapping $f$  from $K\times K \times \cdots \times K$  into $ P^*(K)$ where $K$ appears $n$ times, is called an $n$-ary hyperoperation, and $n$ is called
the arity of this hyperoperation. In this case, an algebraic system $(K, f)$ is called an $n$-ary hypergroupoid or an $n$-ary hypersystems. By  using the  notation $u^j_i$ for the sequence $u_i, u_{i+1},\cdots, u_j$,   $f(u_1,\cdots, u_i, v_{i+1},\cdots, v_j, w_{j+1},\cdots, w_n)$ is written as $f(u^i_1, v^j_{i+1}, w^n_{j+1})$.  We write the expression in the form  $f(u^i_1, v^{(j-i)}, w^n_{j+1})$ if $v_{i+1} =\cdots= v_j = v$.  The  symbol $u^j_i$ is empty when $j< i$. Moreover,  $u_1,\cdots,\widehat{u_i},\cdots,\widehat{u_j}, \cdots,u_n$ indicates that $u_i$ and $u_j$ are eliminated from the sequence $u_1^n$. We  define $f(H^n_1)= \bigcup \{f(u^n_1) \ \vert \ u_i \in H_i, i = 1,\cdots, n \}$ where $\varnothing \neq H_i \subseteq K$ for each $i \in \{1,\cdots,n\}$.
\begin{definition}
\cite{d1} A system  $(K, f, g)$, or simply $K$, is called a commutative Krasner $(m, n)$-hyperring with a scalar identity $1_K$ if it has the following propertices:

\begin{itemize} 
\item[\rm{{\bf 1.}}]~ $(K, f$) is a canonical $m$-ary hypergroup, 

\item[\rm{{\bf 2.}}]~ $(K, g)$ is a $n$-ary semigroup, 

\item[\rm{{\bf 3.}}]~ $g(u^{i-1}_1, f(v^m _1 ), u^n _{i+1}) = f(g(u^{i-1}_1, v_1, u^n_{ i+1}),\cdots, g(u^{i-1}_1, v_m, u^n_{ i+1}))$ for each   $i \in \{1,\cdots,n\}$ and $u^{i-1}_1 , u^n_{ i+1}, v^m_ 1 \in K$. 

\item[\rm{{\bf 4.}}]~ $g(0, u^n _2) = g(u_2, 0, u^n _3) = \cdots = g(u^n_ 2, 0) = 0$ for each $u^n_ 2 \in K$, 

\item[\rm{{\bf 5.}}]~ $f(u_1^n) = f(u_{\sigma(1)}^{\sigma(n)})$ for each $ \sigma \in \mathbb{S}_n$, the group of all permutations of $\{1, \cdots, n\}$ and for each $u_1^n \in K$, 

\item[\rm{{\bf 6.}}]~ $g(u,1_K^{(n-2)})=u$ for all $u \in K$.
\end{itemize}
\end{definition}
Throughout this paper, $K$ is  a commutative Krasner $(m,n)$-hyperring with  identity  $1_K$.
\begin{definition}
Assume that $\varnothing \neq P \subseteq K$. $P$ refers to 
\begin{itemize} 
\item[\rm{{\bf i.}}]~  a subhyperring of $K$ if $(P, f, g)$ is a Krasner $(m, n)$-hyperring. 
\item[\rm{{\bf ii.}}]~  a hyperideal of $K$ if $(P, f)$ is an $m$-ary subhypergroup
of $(K, f)$ and $g(u^{i-1}_1, P, u_{i+1}^n) \subseteq P$, for  each $u^n _1 \in K$ and $i \in \{1,\cdots,n\}$.
\end{itemize} 
\end{definition}

\begin{definition} \cite{sorc1} 
For every element $u \in K$, the hyperideal generated by  $u$  is denoted by $\langle a \rangle$ and defined as $\langle u \rangle=g(K,u,1^{(n-2)})=\{g(a,u,1_K^{(n-2)}) \ \vert \ a \in K\}.$
Moreover, we  define $(P:u)=\{a \in K \ \vert \ g(a,u,1_K^{(n-2)}) \in P\}$ for each hyperideal $P$  of $A$. 
\end{definition}
%\begin{definition} \cite{sorc1}
%A hyperideal $M$ of a Krasner $(m, n)$-hyperring $R$
%is said to be maximal if for every hyperideal $N$ of $R$, $M \subseteq N \subseteq R$ implies that $N=M$ or $N=R$.
%\end{definition}
%The Jacobson radical of a Krasner $(m, n)$-hyperring $R$
%is the intersection of all maximal hyperideals of $R$ and it is denoted by $J_{(m,n)}(R)$. If $R$ does not have any maximal hyperideal, we let $J_{(m,n)}(R)=R$.

\begin{definition} \cite{sorc1}
An element $u \in K$ is  invertible if there exists $v \in K$ such that $1_K=g(u,v,1_K^{(n-2)})$. 
\end{definition}
\begin{definition}
\cite{sorc1} A proper hyperideal  $P$ of $K$ refers to an  $n$-ary  prime hyperideal  if for hyperideals $P_1^n$ of $K$, $g(P_1^ n) \subseteq P$ implies that $P_i \subseteq P$ for some $i \in \{1,\cdots,n\}$.
\end{definition}
 In Lemma 4.5 in \cite{sorc1}, it was shown that   a proper hyperideal $P$ of $K$ is $n$-ary  prime if  $g(u^n_ 1) \in P$ and  $u^n_ 1 \in K$ imply that $a_i \in P$ for some $i \in \{1,\cdots,n\}$.

\begin{definition} \cite{sorc1} 
For a hyperideal $P$ of $K$, the radical of $P$, denoted by $rad(P)$, is the intersection of  all $n$-ary prime hyperideals $K$  containing   $P$. If $K$ does not have any $n$-ary prime hyperideal containing $P$, we define $rad(P)=K$.
\end{definition}
Theorem 4.23 in \cite{sorc1} indicates   that if $u \in rad(P)$ then 
there exists $r \in \mathbb {N}$ such that $g(u^ {(r)} , 1_K^{(n-r)} ) \in P$ for $r \leq n$, or $g_{(x)} (u^ {(r)} ) \in P$ for $r = x(n-1) + 1$.
\begin{definition}
\cite{sorc1} A proper hyperideal  $P$ of $K$ refers to an  $n$-ary  primary hyperideal  if $g(u^n _1) \in P$ and $u_1^ n \in K$ imply that $u_i \in P$ or $g(u_1^{i-1}, 1_K, u_{ i+1}^n) \in rad(P)$ for some $i \in \{1,\cdots,n\}$.
\end{definition}
 The radical of an $n$-ary primary hyperideal $P$ of $K$ is $n$-ary prime (Theorem 4.28 in \cite{sorc1} ).
\begin{definition} \cite{sorc1}
A  Krasner $(m,n)$-hyperring $K$ is called an $n$-ary hyperintegral domain if $K$ is commutative and $g(u_1^n) =0$ implies that $u_i=0$ for some $i \in \{1,\cdots,n\}$.
\end{definition}
\begin{definition} \cite{sorc1}
Let $\varnothing \neq S \subseteq K$. The set $S$ is called  an  $n$-ary multiplicative set if $g(s_1^n) \in S$ for all $s_1^n \in S$.
\end{definition}

%333333333333333333333333333333333333333333333
%33333333333333333333333333333333333333333333333
%3333333333333333333333333333333333333333333333
%33333333333333333333333333333333333333333333333
\section{$\phi$-$\delta$-$S$-primary hyperideals}
A function $\phi$ which assigns to each hyperideal $P$ of $K$ a hyperideal $\phi(P)$ of $K$ is called a reduction  function of hyperideals of $K$ if $\phi(P )\subseteq P$  and for each hperideal $Q$ of $K$ with $P \subseteq Q$, we have $\phi(P) \subseteq \phi(Q)$. For instance, the functions $\phi_0, \phi_1, \phi_n$ and $\phi_w$ from the set of all hyperideals of $K$ to the same set, defined by $\phi_0(P)=\{0\}, \phi_1(P)=P,\phi_n(P)=g(P^{(n)})$ and $\phi_w(P)=\cap_{i=1}^{\infty} g(P^{(i)})$ are reduction  functions of hyperideals of $K$.
Moreover, a function $\delta$ which assigns to each hyperideal $P$ of $K$ a hyperideal $\delta(P)$ of $K$ is called an expansion function of hyperideals of $K$ if $P \subseteq \delta(P)$  and for each hperideal $Q$ of $K$ with $P \subseteq Q$, we have $\delta(P) \subseteq \delta(Q)$. For example, the functions $\delta_0, \delta_1, \delta_K$ and $\delta_M$ from the set of all hyperideals of $K$ to the same set, defined by $\delta_0(P)=P, \delta_1(P)=rad(P),\delta_K(P)=K$ and $\delta_M(P)=\cap_{P \subseteq Q \in Max(K)}Q$ are  expansion  functions of hyperideals of $K$ \cite{mah3}. Now, we introduce the notion of $n$-ary $\phi$-$\delta$-$S$-primary hyperideals extending $S$-primary  hyperideals into a  general framework with the help of the reduction and expansion function. The following
definition constitutes the $\phi$-$\delta$-version of $S$-primary hyperideals.
\begin{definition}
Let $S$ be an $n$-ary multiplicative subset of $K$ and $P$ be a proper hyperideal of  $K$ disjoint with $S$.  Assume that $\phi$ and $\delta$ are reduction and expansion functions of hyperideals of $K$, respectively.
%\begin{enumerate} 
%\item $P$ is called a $\delta$-$S$-primary hyperideal of $K$ associated to  $s \in S$, if whenever $g(u_1^n) \in P$, then $g(u_i,s,1_K^{(n-2)}) \in P$ or $g(u_1^{i-1},s,u_{i+1}^n) \in \delta(P)$ for all $u_1^n \in K$.
%\item
 $P$ is called an $n$-ary  $\phi$-$\delta$-$S$-primary hyperideal of $K$ associated to  $s \in S$, if whenever $g(u_1^n) \in P \backslash \phi(P)$, then $g(u_i,s,1_K^{(n-2)}) \in P$ or $g(u_1^{i-1},s,u_{i+1}^n) \in \delta(P)$ for all $u_1^n \in K$.
%\end{enumerate} 
\end{definition}
\begin{example} 
Consider the Krasner $(2,2)$-hyperring $K=\{0,1,u\}$ with the hyperaddition and multiplication defined by
$\vspace{0.5cm}$

$\hspace{3cm}$
$\begin{tabular}{|c||c|c|c|} 
\hline $+$ & $0$ & $1$ & $u$
\\ \hline \hline $0$ & $0$ & $1$ & $x$
\\ \hline $1$ & $1$ & $K$ & $1$ 
\\ \hline $u$ & $u$ & $1$ & $\{0,u\}$ 
\\ \hline
\end{tabular}$
$\hspace{0.5cm}$
$\begin{tabular}{|c||c|c|c|} 
\hline $\cdot$ & $0$ & $1$ & $u$
\\ \hline \hline $0$ & $0$ & $0$ & $0$
\\ \hline $1$ & $0$ & $1$ & $u$ 
\\ \hline $u$ & $0$ & $u$ & $0$ 
\\ \hline
\end{tabular}$
$\vspace{0.5cm}$

 Then the hyperideal $P=\{0,u\}$ is a $2$-ary  $\phi_2$-$\delta_0$-$\{0\}$-primary  hyperideal of $K$.
\end{example}

\begin{example} \label{exa2}
Consider the Krasner $(2,3)$-hyperring $(K=[0,1],\oplus,g )$ where the 2-ary hyperoperation $``\oplus"$ is defined as $u \oplus v=\{\max\{u, v\}\}$ if $u \neq v$ or $[0,u]$ if $u =v.$ 
Also, $``g"$ is the usual multiplication on real numbers. Assume that $S=\{1\}$. Then the hyperideal $P=[0,0.5]$ is a $3$-ary  $\phi_w$-$\delta_1$-$S$-primary hyperideal of $K$ but it is not $3$-ary $\phi_w$-$\delta_0$-$S$-primary since $g(0.6,0.7,0.8) \in P \backslash \phi_w(P)$ and $0.6,0.7,0.8 \notin P$. Now, let $T=(0,0.1]$. Then $P$ is a $3$-ary $\phi_w$-$\delta_0$-$T$-primary.

\end{example}
\begin{remark}
Every $\delta$-primary  hyperideal of $K$ is a $\phi$-$\delta$-$S$-primary  hyperideal. 
\end{remark}
The following example shows that the converse does not necessarily hold.
\begin{example}
Assume that $K$ is the Krasner $(2,4)$-hyperring given in Example 4.8 in \cite{sorc1}. In \cite{davvazz}, it was shown that $\langle 0 \rangle$ is not a $\delta_0$-primary hyperideal of $K$. Now, let $S=\{2,3\}$. 
Then $\langle 0 \rangle$ is a $\phi_4$-$\delta_0$-$S$-primary hyperideal of $K$. 
\end{example}
%\begin{remark} \label{exa} 
%Every $\phi$-primary hyperideal of $K$ is $\phi$-$\delta$-$S$-primary.
%\end{remark} 
%The converse may not be always true as it is shown in the following example. 

In the sequel, we assume that $\phi$ and $\delta$ are reduction and expansion functions of hyperideals of $K$, respectively, and $S$ is an $n$-ary multiplicative subset of $K$. 
\begin{theorem} \label{1}
Let $P$ be a proper hyperideal of $K$. If $P$ is an $n$-ary $\phi$-$\delta$-$S$-primary hyperideal of $K$ associated to $s \in S$ such that $rad(\phi(P)) \subseteq \phi(rad(P))$ and $rad(\delta(P)) \subseteq \delta(rad(P))$, then $rad(P)$ is an $n$-ary $\phi$-$\delta$-$S$-primary hyperideal of $K$ associated to $s \in S$.
\end{theorem}
\begin{proof}
From $S \cap P = \varnothing$ it follows that $S \cap rad(P) =\varnothing$. Assume that $g(u_1^n) \in rad(P) \backslash
 \phi(rad(P))$ for $u_1^n \in K$. This means that there exists $r \in \mathbb {N}$ such that $g(g(u_1^n)^ {(r)} , 1_K^{(n-r)} ) \in P$ for $r \leq n$, or $g_{(x)} (g(u_1^n)^ {(r)} ) \in P$ for $r = x(n-1) + 1$. In the former case, if $g(g(u_1^n)^ {(r)} , 1_K^{(n-r)} ) \in \phi(P)$, then we get $g(u_1^n) \in rad(\phi(P)) \subseteq \phi(rad(P))$ which is impossible. In the second case, if $g_{(x)} (g(u_1^n)^ {(r)} ) \in \phi(P)$, then it leads to a contradiction by a similar argument. Then we have $g(g(u_1^n)^ {(r)} , 1_K^{(n-r)} ) \in P \backslash \phi(P)$ for $r \leq n$, or $g_{(x)} (g(u_1^n)^ {(r)} ) \in P \backslash \phi(P)$ for $r = x(n-1) + 1$. In the first possibility, since $P$ is a $\phi$-$\delta$-$S$-primary hyperideal of $K$ associated to $s \in S$ and $g(g(u_i^{(r)},1_K^{(n-r)}),g(g(u_1^{i-1},1_K,u_{i+1}^n)^{(r)},1_K^{(n-r)}),1_K^{(n-2)}) \in P \backslash \phi(P)$, we get the result that $g(s,g(u_i^{(r)},1_K^{(n-r)}),1_K^{(n-2)})=g(s,u_i^{(r)},1_K^{(n-r-1)}) \in P $ which means $g(g(s,u_i,1_K^{(n-2)})^r,1_K^{(n-r)}) \in P$ or $g(s,g(u_1^{i-1},1_K,u_{i+1}^n)^{(r)},1_K^{(n-r-1)}) \in \delta(P)$ which implies $g(g(u_1^{i-1},s,u_{i+1}^n)^r,1_K^{(n-r)}) \in \delta(P)$.  Hence we conclude that $g(s,u_i,1_K^{(n-2)}) \in rad(P)$ or $g(u_1^{i-1},s,u_{i+1}^n) \in rad(\delta(P)) \subseteq \delta(rad(P))$. Thus $rad(P)$ is a $\phi$-$\delta$-$S$-primary hyperideal of $K$ associated to $s \in S$. In the second case, one can easily complete the proof by using a similar argument.
 \end{proof}
Note that the conditions $``rad(\phi(P)) \subseteq \phi(rad(P))"$ and $``rad(\delta(P)) \subseteq \delta(rad(P))"$ in Theorem \ref{1} can not be ignored. 
 \begin{example}\label{ignor}
Assume that $A=\mathbb{Z}_3[X,Y,Z]$ and $P=\langle X^3Y^3Z^3\rangle$. Then $K=A/P$ is a Krasner $(m,n)$-hyperring with ordinary addition and ordinary multiplication. Let $S=\{1+P\}$. Then $P/P$ is a $\phi_0$-$\delta_0$-$S$-primary hyperideal of $K$. Note that $rad(\phi_0(P/P)) \nsubseteq \phi_0(rad(P/P))$ and  
$rad(P/P)$ is not a $\phi_0$-$\delta_0$-$S$-primary hyperideal of $K$ because $2XYZ+P=(2X+P)(Y+P)(Z+P) \in rad(P/P) \backslash \phi_0(rad(P/P))$ but none of the elements $(2X+P), (Y+P)$ and $(Z+P)$ are not in $\delta_0(rad(P/P))$.
\end{example}
A hyperideal $P$ of $K$ is said to be an $n$-ary  $\phi$-$S$-primary hyperideal associated to  $s \in S$, if whenever $g(u_1^n) \in P \backslash \phi(P)$, then $g(u_i,s,1_K^{(n-2)}) \in P$ or $g(u_1^{i-1},s,u_{i+1}^n) \in rad(P)$ for all $u_1^n \in K$. Now, in view of Theorem \ref{1}, we have the following result.
\begin{corollary} \label{2}
Let $P$ be a proper hyperideal of $K$. If $P$ is an $n$-ary $\phi$-$S$-primary hyperideal of $K$ associated to $s \in S$ such that $rad(\phi(P)) \subseteq \phi(rad(P))$, then $rad(P)$ is an $n$-ary $\phi$-$S$-prime hyperideal of $K$ associated to $s \in S$.
\end{corollary}
\begin{proof}
It suffices to take $\delta(I)=rad(I)$ for every hyperideal $I$ of $K$ in Theorem \ref{1}.
\end{proof}
\begin{theorem}\label{3}
Let $P$ be an $n$-ary $\phi$-$S$-primary hyperideal of $K$ associated to $s \in S$ such that $rad(\phi(P)) \subseteq \phi(rad(P))$ and $(\phi(rad(P)) : t) \subseteq (\phi(rad(P)) : s)$ for all $t \in S$. If $u \in K \backslash (rad(P) : s)$, then $(rad(P) : u) \cap S=\varnothing$.
\end{theorem}
\begin{proof}
Since $P \cap S=\varnothing$, we conclude that $rad(P) \cap S=\varnothing$. By corollary \ref{2}, $rad(P)$ is a $\phi$-$S$-prime hyperideal of $K$ associated to $s \in S$. Assume that $a \in (rad(P) : u) \cap S$. This implies that $g(a,u,1_K^{(n-2)}) \in rad(P)$. Let $g(a,u,1_K^{(n-2)}) \in \phi(rad(P))$. It follows that $u \in (\phi(rad(P)) : a) \subseteq (\phi(rad(P)) : s)$. Then we get $g(s,u,1_K^{(n-2)}) \in \phi(rad(P)) \subseteq rad(P)$ which is impossible. Therefore, we have $g(a,u,1_K^{(n-2)}) \in rad(P) \backslash \phi(rad(P))$. Since $rad(P)$ is an $n$-ary $\phi$-$S$-prime hyperideal of $K$, we obtain $g(s,a,1_K^{(n-2)}) \in rad(P)$ which contradicts the fact that $rad(P) \cap S=\varnothing$ or $g(s,u,1_K^{(n-2)}) \in rad(P)$ which contradicts the assumption that $u \notin  (rad(P) : s)$. Then we conclude that $(rad(P) : u) \cap S=\varnothing$.
\end{proof}
We give the following results obtained by the previous theorem.
\begin{corollary}\label{4}
Let $P$ be an $n$-ary $\phi$-$\delta$-$S$-primary hyperideal of $K$ associated to $s \in S$ such that $(\delta(P) : s) =(rad(P) : s)$,  $(\phi(rad(P)) : t) \subseteq (\phi(rad(P)) : s)$ for all $t \in S$ and $\delta(P) \subseteq rad(P)$. Then $(\delta(P) : s)=(\delta(P) : g(s^n))$. In particular, if $u \in K \backslash (\delta(P) : s)$, then $(\delta(P) : u) \cap S=\varnothing.$ 
\end{corollary}
\begin{proof}
Since $\delta(P) \subseteq rad(P)$ and $P$ is an $n$-ary $\phi$-$\delta$-$S$-primary hyperideal of $K$ associated to $s \in S$, we conclude that $(rad(P) : s)=(rad(P) : g(s^n))$ and $P$ is a $\phi$-$S$-primary hyperideal of $K$ associated to $s \in S$. Since $(\delta(P) : s) =(rad(P) : s)$ and $\delta(P) \subseteq rad(P)$, we get the result that $(\delta(P) : s)=(\delta(P) : g(s^n))$. Furthermore, let $u \notin (\delta(P) : s)$. Then we obtain $g(s,u,1_K^{(n-2)}) \notin rad(P)$. Therefore we conclude that $(rad(P) : u) \cap S=\varnothing$ by Theorem \ref{3}. Since $\delta(P) \subseteq rad(P)$, we get the result that $(\delta(P) : u) \cap S \subseteq (rad(P) : u) \cap S=\varnothing$. 
\end{proof}
\begin{theorem} \label{5}
Let $P$ be an $n$-ary $\phi$-$\delta$-$S$-primary hyperideal of $K$ associated to $s \in S$. If $u \in K \backslash (\delta(P) : g(s^n))$, then $(P : g(s,u,1_K^{(n-2)}))=(\phi(P) : g(s,u,1_K^{(n-2)}))$ or $(P : g(s,u,1_K^{(n-2)}))=(P : s)$.
\end{theorem}
\begin{proof}
Let $a \in (P : g(s,u,1_K^{(n-2)}))$. Therefore we have $g(a,g(s,u,1_K^{(n-2)}),1_K^{(n-2)}) \in P$. If $g(a,g(s,u,1_K^{(n-2)}),1_K^{(n-2)}) \in \phi(P)$, then $a \in (\phi(P) : g(s,u,1_K^{(n-2)}))$ which means $(P : g(s,u,1_K^{(n-2)})) \subseteq (\phi(P) : g(s,u,1_K^{(n-2)}))$. On the other hand, we have $(\phi(P) : g(s,u,1_K^{(n-2)})) \subseteq (P : g(s,u,1_K^{(n-2)})) $ as $\phi(P) \subseteq P$. Hence we conclude that $(P : g(s,u,1_K^{(n-2)}))=(\phi(P) : g(s,u,1_K^{(n-2)}))$. Now, let $g(a,g(s,u,1_K^{(n-2)}),1_K^{(n-2)}) \notin \phi(P)$, so $g(a,g(s,u,1_K^{(n-2)}),1_K^{(n-2)}) \in P \backslash \phi(P)$. Since $P$ is  a $\phi$-$\delta$-$S$-primary hyperideal of $K$ associated to $s \in S$ and $u \notin (\delta(P) : g(s^n))$, we get the result that $g(s,a,1_K^{(n-2)}) \in P$ which implies $a \in (P : s)$. Thus $(P : g(s,u,1_K^{(n-2)})) \subseteq (P : s)$. Also, we know that $  (P : s) \subseteq (P : g(s,u,1_K^{(n-2)}))$. Consequently, $(P : g(s,u,1_K^{(n-2)}))=(P : s)$.
\end{proof}
Now, we have the following result.
\begin{corollary}\label{6}
Let $P$ be an $n$-ary $\phi$-$S$-primary hyperideal of $K$ associated to $s \in S$. If $u \in K \backslash (rad(P) : s)$, then $(P : g(s,u,1_K^{(n-2)}))=(\phi(P) : g(s,u,1_K^{(n-2)}))$ or $(P : g(s,u,1_K^{(n-2)}))=(P : s)$.
\end{corollary}
\begin{proof}
Assume that  $P$ is an $n$-ary $\phi$-$S$-primary hyperideal of $K$ associated to $s \in S$. Since $(rad(P) : s)=(rad(P) : g(s^n))$, we conclude that  $u \notin (rad(P) : g(s^n))$ by the hypothesis. Thus we get the result that $(P : g(s,u,1_K^{(n-2)}))=(\phi(P) : g(s,u,1_K^{(n-2)}))$ or $(P : g(s,u,1_K^{(n-2)}))=(P : s)$ by Theorem  \ref{5}.
\end{proof}
\begin{theorem}
Let $P$ be an  $n$-ary $\phi$-$\delta$-$S$-primary hyperideal of $K$ associated to $s \in S$ such that $(\phi(P) : u) \subseteq \phi(P : u)$ for every $u \notin P$. Then so is $(P : u)$.
\end{theorem}
\begin{proof}
Let $g(u_1^n) \in (P : u) \backslash \phi((P : u))$ for $u_1^n \in K$. Therefore we conclude that $g(g(u,u_i,1_K^{(n-2)}),g(u_1^{i-1},1_K,u_{i+1}^n),1_K^{(n-2)})=g(g(u_1^n),u,1_K^{(n-2)}) \in P \backslash \phi(P)$. Since $P$ is an  $n$-ary $\phi$-$\delta$-$S$-primary hyperideal of $K$ associated to $s \in S$, we get the result that $g(u,g(s,u_i,1_K^{(n-2)}),1_K^{(n-2)})
=g(s,g(u,u_i,1_K^{(n-2)}),1_K^{(n-2)}) \in P$ which means $g(s,u_i,1_K^{(n-2)}) \in (P : u)$ or $g(u_1^{i-1},s,u_{i+1}^n))=
g(s,g(u_1^{i-1},1_K,u_{i+1}^n),1_K^{(n-2)}) \in \delta(P) \subseteq (\delta(P) : u)$. This implies that $(P : u)$ is an  $n$-ary $\phi$-$\delta$-$S$-primary hyperideal of $K$ associated to $s \in S$
\end{proof}
\begin{proposition}\label{}
Let  $Q$ be a hyperideal of $K$, $P$  an $n$-ary $\phi$-$\delta$-$S$-primary hyperideal of $K$ associated to $s \in S$, $\phi(P) = \phi(I)$ for every hyperideal $I$ of $K$ and $Q \cap S \neq \varnothing $. If $\delta(P \cap Q)=\delta(P) \cap \delta(Q)$,  then $P \cap Q$ is an $n$-ary $\phi$-$\delta$-$S$-primary hyperideal of $K$.
\end{proposition}
\begin{proof}
Clearly,   $(P \cap Q) \cap S=\varnothing$. Assume that  $u_1^n \in K$ such that  $ g(u_1^n) \in (P \cap Q) \backslash \phi(P \cap Q)$. Since $P$  is an $n$-ary $\phi$-$\delta$-$S$-primary hyperideal of $K$ associated to $s \in S$ and $g(u_1^n) \in P \backslash \phi(P) $, we conclude that  $g(s,u_i,1_K^{(n-2)}) \in P$  or $g(u_1^{i-1},s,u_{i+1}^n) \in \delta(P)$ for some $i \in \{1,\cdots,n\}$. Now, take any $t \in Q \cap S$. So $g(t^{(n-1)},s) \in S$. Then we get the result that  $g(g(t^{(n-1)},s),u_i,1_K^{(n-2)})=g(t^{(n-1)},g(s,u_i,1_K^{(n-2)})) \in P \cap Q$ or $g(u_1^{i-1},g(t^{(n-1)},s),u_{i+1}^n)=g(g(t^{(n-1)}, g(u_1^{i-1},s,u_{i+1}^n)) \in \delta(P) \cap \delta(Q)=rad(P \cap Q)$. This means that  $P \cap Q$ is an $n$-ary $\phi$-$\delta$-$S$-primary hyperideal of $K$.
\end{proof}
\begin{proposition} \label{}
Let $P_1^{n-1}$   be some hyperideals of $K$, $P$  an $n$-ary $\phi$-$\delta$-$S$-primary hyperideal of $K$ associated to $s \in S$, $\phi(P) = \phi(I)$ for every hyperideal $I$ of $K$ and $\delta(g(P_1^{n-1},P))=\delta(P_1) \cap \cdots \cap \delta(P_{n-1}) \cap \delta(P) $. If $P_j \cap S \neq \varnothing $ for each $j \in \{1,\cdots,n-1\}$,  then $g(P_1^{n-1},P)$ is an $n$-ary $\phi$-$\delta$-$S$-primary hyperideal of $K$.
\end{proposition}
\begin{proof}
It is obvious that  $g(P_1^{n-1},P) \cap S=\varnothing$.  Suppose that $u_1^n \in K$ such that $ g(u_1^n) \in g(P_1^{n-1},P) \backslash \phi(g(P_1^{n-1},P))$. Since $ g(u_1^n) \in P \backslash \phi(P) $ and $P$ is an $n$-ary $\phi$-$\delta$-$S$-primary hyperideal of $K$ associated to $s \in S$, we conclude that $g(s,u_i,1_K^{(n-2)}) \in P$ or $g(u_1^{i-1},s,u_{i+1}^n) \in \delta(P)$ for some $i \in \{1,\cdots,n\}$. Now, take  $t_j \in  P_j \cap S$ for any $j \in \{1,\cdots,n-1\}$. Therefore we get  $g(t_1^{n-1},s) \in S$. If $g(s,u_i,1_K^{(n-2)}) \in P$, then  $g(g(t_1^{n-1},s),u_i,1_K^{n-2})=g(t_1^{n-1},g(s,u_i,1_K^{(n-2)})) \in g(P_1^{n-1},P)$. If $g(u_1^{i-1},s,u_{i+1}^n) \in \delta(Q)$,  then we get $g(u_1^{i-1},
g(t_1^{n-1},s),u_{i+1}^n))=g(t_1^{n-1},g(u_1^{i-1},s,u_{i+1}^n)) \in \delta(P_1) \cap \cdots \cap \delta(P_{n-1}) \cap \delta(P)=\delta(g(P_1^{n-1},P))$. Thus, $g(P_1^{n-1},P)$ is an $n$-ary $\phi$-$\delta$-$S$-primary hyperideal of $K$.
\end{proof}
Recall from \cite{mah3} that a proper hyperideal $P$ of $K$ is called $n$-ary $\delta$-primary if for all $u_1^n \in K$, $g(u_1^n) \in P$  implies that  $u_i \in P$ or $g(u_1^{i-1},1_K,u_{i+1}^n) \in \delta(P)$ for some $i \in \{1,\cdots,n\}$. Furthermore, a hyperideal $P$ of $K$ is said to be a $\delta$-$S$-primary hyperideal associated to  $s \in S$, if whenever $g(u_1^n) \in P$, then $g(u_i,s,1_K^{(n-2)}) \in P$ or $g(u_1^{i-1},s,u_{i+1}^n) \in \delta(P)$ for all $u_1^n \in K$.
\begin{theorem}\label{hadi}
Let $\phi(I)=\phi^2(I)$ for every hyperideal $I$ of $K$. Then every $n$-ary  $\phi$-$\delta$-$S$-primary hyperideal of $K$ is  $n$-ary $\delta$-primary if and only if  $\phi(P)$ is an $n$-ary $\delta$-primary hyperideal of $K$ for each hyperideal $P$ in $K$ and every $n$-ary $\delta$-$S$-primary hyperideal of $K$ is $n$-ary $\delta$-primary.
\end{theorem}
\begin{proof}
$\Longrightarrow$ Since $\phi(P)=\phi^2(P)$, we get the result that $\phi(P)$ is an $n$-ary $\phi$-$\delta$-$S$-primary hyperideal of $K$. By the hypothesis, we conclude that $\phi(P)$ is an $n$-ary $\delta$-primary hyperideal of $K$.  Let $Q$ be an $n$-ary $\delta$-$S$-primary hyperideal of $K$. This implies that $Q$ is an $n$-ary  $\phi$-$\delta$-$S$-primary hyperideal of $K$ which means $Q$ is $n$-ary $\delta$-primary, by the assumption.

$\Longleftarrow$ Assume that $Q$ is an $n$-ary $\phi$-$\delta$-$S$-primary hyperideal of $K$ associated to $s \in S$. Let $g(u_1^n) \in Q$ for $u_1^n \in K$. If $g(u_1^n) \notin \phi(Q)$, then we have $g(u_1^n) \in Q \backslash \phi(Q)$. It follows that $g(s,u_i,1_K^{(n-2)}) \in Q$ or $g(u_1^{i-1},s,u_{i+1}^n) \in \delta(Q)$ for some $i \in \{1,\cdots,n\}$. This means that $Q$ is an $n$-ary $\delta$-$S$-primary hyperideal of $K$ and so $Q$ is an $n$-ary $\delta$-primary hyperideal. Now, let $g(u_1^n) \in \phi(Q)$. Since $\phi(Q)$ is an $n$-ary $\delta$-primary hyperideal of $K$,  we obtain $u_i \in \phi(Q) \subseteq Q$ for some $i \in \{1,\cdots,n\}$ which means $g(s,u_i,1_K^{(n-2)}) \in Q$ or $g(u_1^{i-1},1_K,u_{i+1}^n) \in \delta(\phi(Q))$ which means $g(u_1^{i-1},s,u_{i+1}^n) \in \delta(\phi(Q)) \subseteq \delta(Q)$. This implies that $Q$ is an $n$-ary  $\delta$-$S$-primary hyperideal of $K$ and so $Q$ is an $n$-ary $\delta$-primary hyperideal of $K$.
\end{proof}
Assume that  $P$ is a proper hyperideal of $K$. $P$ is called an $n$-ary $S$-primary hyperideal of $K$ associated to $s \in S$ if  for all $u_1^n \in K$, $g(u_1^n) \in P$ implies  $g(s,u_i,1_K^{(n-2)}) \in P$ or $g(u_1^{i-1},s,u_{i+1}^n) \in rad(P)$ for some $i \in \{1,\cdots,n\}$ \cite{mah6}. 
Moreover,  $P$ refers to a weakly $n$-ary $S$-primary hyperideal associated to $s \in S$  if  for all $u_1^n \in K$ with $0 \neq g(u_1^n) \in P$, we have $g(s,u_i,1_K^{(n-2)}) \in P$ or $g(u_1^{i-1},s,u_{i+1}^n) \in rad(P)$ for some $i \in \{1,\cdots,n\}$ \cite{mah8}. Now, we have the following
result as a consequence of Theorem \ref{hadi}.
\begin{corollary}
Every weakly $n$-ary $S$-primary hyperideal of $K$ is primary if and only if $K$ is an $n$-ary  hyperintegral domain and every $n$-ary $S$-primary hyperideal of $K$ is an $n$-ary primary hyperideal.
\end{corollary}
\begin{proof}
By taking $\delta=\delta_1$ and $\phi=\phi_0$ in Theorem \ref{hadi}, we are done.
\end{proof}
\begin{proposition} \label{2pezeshk}
Assume that  $S \subseteq T$ are two $n$-ary multiplicative subsets of $K$ such that for each $s \in T$, there exists $s^\prime \in T$ with $g(s^{(n-1)},s^{\prime}) \in S$. If $P$ is an  $n$-ary $\phi$-$\delta$-$T$-primary hyperideal of $K$ associated to $s \in T$, then $P$ is an $n$-ary $\phi$-$\delta$-$S$-primary hyperideal of $P$.
\end{proposition}
\begin{proof}
Assume that $u_1^n \in K$ such that $ g(u_1^n) \in P \backslash \phi(P)$. Since $P$ is an  $n$-ary $\phi$-$\delta$-$T$-primary hyperideal of $K$ associated to $s \in T$,  we have $g(s,u_i,1_K^{(n-2)}) \in P$ or $g(u_1^{i-1},s,u_{i+1}^n) \in \delta(P)$ for some $i \in \{1,\cdots,n\}$. By the assumption, there exists $s^{\prime} \in T$ such that $g(s^{(n-1)},s^\prime) \in S$. Put $t=g(s^{(n-1)},s^\prime)$. Let  $g(s,u_i,1^{(n-2)}) \in P$. Then we get the result that   $g(t,u_i,1_K^{(n-2)})=g(g(s^{(n-1)},s^{\prime}),u_i,1_K^{(n-2)})
=g(g(s^{(n-2)},s^{\prime},1_K),g(s,u_i,1^{(n-2)}),1_K^{(n-2)}) \in P$. Now, assume that $g(u_1^{i-1},s,u_{i+1}^n) \in \delta(P)$. Therefore we conclude that  $g(u_1^{i-1},t,u_{i+1}^n)=g(u_1^{i-1},g(s^{(n-1)},s^{\prime}),u_{i+1}^n)
=g(g(s^{n-2},s^{\prime},1_A),g(u_1^{i-1},s,u_{i+1}^n),1_K^{(n-2)}) \in \delta(P)$. Thus $P$ is an $n$-ary $\phi$-$\delta$-$S$-primary hyperideal of $P$.
\end{proof}
The notion of Krasner $(m,n)$-hyperring of fractions was studied in \cite{mah5}.
\begin{theorem} \label{}
Let  $P$ be a hyperideal of $A$ and $1_K \in S$. Then $P$ is an $n$-ary $\phi$-$\delta$-$S$-primary  hyperideal of $K$ if and only if $P$ is an  $n$-ary $\phi$-$\delta$-$S^{\prime}$-primary  hyperideal where $S^{\prime}=\{a \in K \ \vert \ \frac{a}{1_K} \ \text{is invertible in} \ S^{-1}K \}$.
\end{theorem}
\begin{proof}
$\Longrightarrow$ Since $S^{\prime}$ is an $n$-ary multiplicative subset of $K$ containing $S$ and $P$ is an $n$-ary $\phi$-$\delta$-$S$-primary  hyperideal of $K$, we are done.

$\Longleftarrow$ Assume that $s \in S^{\prime}$. This implies that $\frac{s}{1_K}$ is invertible in $S^{-1}K$. Then there exists $u \in K$ and $t \in S$ such that $G(\frac{s}{1_A},\frac{u}{t},\frac{1_K}{1_K}^{(n-2)})=\frac{g(s,u,1_K^{(n-2)})}{g(t,1_K^{(n-1)})}=\frac{1_K}{1_K}$.  So there exists $t^{\prime} \in S$ with $0 \in g(t^{\prime}, f(g(s,u,1_K^{(n-2)}),-g(t,1_K^{(n-1)}),0^{(m-2)}),1_K^{(n-2)})=f(g(t^{\prime}, s,u,1_K^{(n-3)}),-g(t^{\prime},t,1_K^{(n-2)}),0^{(m-2)})$. 
Since  $g(t^{\prime},t,1_K^{(n-2)}) \in S$, we get the result that $g(t^{\prime}, s,u,1_K^{(n-3)}) \in f(g(t^{\prime},t,1_K^{(n-2)}),0^{(m-1)}) \subseteq S$.
Let  $t^{\prime \prime }=g(t^{\prime},u,1_K^{(n-2)})$. Since $G(\frac{g(t^{\prime},u,1_K^{(n-2)})}{1_A},\frac{g(s,1_K^{(n-1)})}{g(t^{\prime},u,s,1_K^{(n-3)})},\frac{1_K}{1_K}^{(n-2)})=\frac{g(t^{\prime},u,s,1_K^{(n-3)})}{g(t^{\prime},u,s,1_K^{(n-3)})}=\frac{1_K}{1_K}$, we get $t^{\prime \prime } \in S^{\prime}$.  Therefore we obtain $g(s^{(n-1)},g({t^{\prime \prime}}^{(n-1)},g(s,t^{\prime \prime},1_K^{(n-2)})))=g(g(s,t^{\prime \prime},1_K^{(n-2)})^n) \in S$. Put  $s^{\prime}=g({t^{\prime \prime}}^{(n-1)},g(s,t^{\prime \prime},1_K^{(n-2)}))$. So $s^{\prime} \in S^{\prime}$. Since $g(s^{(n-1)},s^{\prime}) \in S$,  we conclude that $P$ is an  $n$-ary $\phi$-$\delta$-$S$-primary  hyperideal of $K$, by Proposition \ref{2pezeshk}.
\end{proof}
A proper hyperideal $P$ of $K$  is called a strongly $n$-ary $\phi$-$\delta$-$S$-primary hyperideal of $K$ associated to  $s \in S$, if whenever $g(P_1^n) \subseteq P \backslash \phi(P)$, then $g(P_i,s,1_K^{(n-2)}) \subseteq P$ or $g(P_1^{i-1},s,P_{i+1}^n) \subseteq \delta(P)$ for all hyperideals $P_1^n$ of $ K$. Note that every strongly $\phi$-$\delta$-$S$-primary hyperideal of $K$ is a $\phi$-$\delta$-$S$-primary hyperideal of $K$.
\begin{theorem} \label{7} 
Let $P$ be a strongly $n$-ary $\phi$-$\delta$-$S$-primary hyperideal of $K$ associated to $s \in S$. If $g(u_1^n) \in \phi(P)$ for $u_1^n \in K$ but $g(s,u_i,1_k^{(n-2)}) \notin P$ and $g(u_1^{i-1},s,u_{i+1}^n) \notin \delta(P)$ for all $i \in \{1,\cdots,n\}$, then for each $j_1,\cdots,j_r \in \{1,\cdots,n\}$, $g(u_1,\cdots,\widehat{u_{j_1}}, \cdots,\widehat{u_{j_2}},\cdots, \widehat{u_{j_r}}, \cdots, P^{(r)}) \subseteq \phi(P)$ . 
\end{theorem} 
\begin{proof}
We use the induction on $r$. Let $r=1$. Suppose on the contrary that $g(u_1^{i-1},P,u_{i+1}^n) \nsubseteq \phi(P)$ for some $i \in \{1,\cdots,n\}$. Then $g(u_1^{i-1},u,u_{i+1}^n) \notin \phi(P)$ for some $u \in P$. Therefore  we get the result that $g(u_1^{i-1},f(u,u_i,0^{(m-2)}),u_{i+1}^n)= f(g(u_1^n),g(u_1^{i-1},u,u_{i+1}^n),0^{(m-2)}) \subseteq P \backslash \phi(P)$. By the hypothesis, we conclude that $g(s,f(u,u_i,0^{(m-2)}),1_K^{(n-2)})=f(g(s,u,1_K^{(n-2)}),g(s,u_i,1_K^{(n-2)}),0^{(m-2)}) \subseteq P$ which implies $g(s,u_i,1_K^{(n-2)}) \in P$ or $g(u_1^{i-1},s,u_{i+1}^n) \in \delta(P)$ for some $i \in \{1,\cdots,n\}$ which are impossible. Now, let the
claim be true for all positive integers being less than $r$. Let $g(u_1,\cdots,\widehat{u_{j_1}}, \cdots,\widehat{u_{j_2}},\cdots, \widehat{u_{j_r}}, \cdots, P^{(r)}) \nsubseteq \phi(P)$ for some $j_1,\cdots,j_r \in \{1,\cdots,n\}$. Without loss of generality, we eliminate $u_1^r$, so $g(u_{r+1},\cdots,u_n,P^{(r)}) \nsubseteq \phi(P)$. Then $ g(u_{r+1},\cdots,u_n,v_1^r) \notin \phi(P)$ for some $v_1^r \in P$. By induction hypothesis, we get $ g(f(u_1,v_1,0^{(m-2)}),\cdots,f(u_r,v_r,0^{(m-2)}),u_{r+1}^n) \subseteq P \backslash \phi(P)$. Since $P$ is a strongly $n$-ary $\phi$-$\delta$-$S$-primary hyperideal of $K$,  $g(s,f(u_i,v_i,0^{(m-2)}),1_K^{(n-2)}) \subseteq P$  or $g(f(u_1,v_1),\cdots,\widehat{f(u_i,v_i,0^{(m-2)})},\cdots,f(u_r,v_r,0^{(m-2)}),s,u_{r+1}^n) \subseteq \delta(P)$  for some $i \in \{1,\cdots,r\}$ or   $g(f(u_1,v_n,0^{(m-2)}),\cdots,f(u_r,v_r,0^{(m-2)}),u_{r+1}^{j-1},s, u_{j+1}^n) \subseteq \delta(Q)$ or $g(s,u_j,1_K^{(n-2)}) \in P$ for some $j \in \{r+1,\cdots,n\}$. This implies that $g(s,u_i,1_K^{(n-2)}) \in P$ or $g(u_1^{i-1},s,u_{i+1}^n) \in \delta(P)$ for some $i \in \{1,\cdots,n\}$ which are  impossible. Thus $g(u_1,\cdots,\widehat{u_{j_1}}, \cdots,\widehat{u_{j_2}},\cdots, \widehat{u_{j_r}}, \cdots, P^{(r)}) \subseteq \phi(P)$ for each $j_1,\cdots,j_r \in \{1,\cdots,n\}$.  
\end{proof}
\begin{theorem} \label{8}
Let $P$ be  a strongly  $n$-ary $\phi$-$\delta$-$S$-primary hyperideal of $K$ associated to $s \in S$ but is not an $n$-ary $\delta$-$S$-primary hyperideal. Then $g(P^{(n)}) \subseteq \phi(P)$.
\end{theorem}
\begin{proof}
Assume that $P$ is  a strongly  $n$-ary $\phi$-$\delta$-$S$-primary hyperideal of $K$ associated to $s \in S$ and $g(P^{(n)}) \nsubseteq \phi(P)$.  Let $u_1^n \in K$ such that $g(u_1^n) \in P$. If $ g(u_1^n) \notin \phi(P)$, then we get $g(s,u_i,1_K^{(n-2)}) \in P$ or $g(u_1^n,s,u_{i+1}^n) \in \delta(P)$ for some $i \in \{1,\cdots,n\}$ which implies $P$ is an $n$-ary $\delta$-$S$-primary hyperideal of $K$, a contradiction. Suppose that $g(u_1^n) \in \phi(P)$. If $g(u_1,\cdots,\widehat{u_{j_1}}, \cdots,\widehat{u_{j_2}},\cdots, \widehat{u_{j_r}}, \cdots, P^{(r)}) \nsubseteq \phi(P)$. Then we conclude that $g(s,u_i,1_K^{(n-2)}) \in P$ or $g(u_1^n,s,u_{i+1}^n) \in \delta(P)$ for some $i \in \{1,\cdots,n\}$ in view of Theorem \ref{7}, a contradiction. Now, assume that $g(u_1,\cdots,\widehat{u_{j_1}}, \cdots,\widehat{u_{j_2}},\cdots, \widehat{u_{j_r}}, \cdots, P^{(r)}) \subseteq \phi(P)$ for all $r \in \{1,\cdots,n-1\}$.
 Since $ g(P^{(n)}) \nsubseteq \phi(P)$, there exist $v_1^n \in P$ such that $g(v_1^n) \notin \phi(P)$. Therefore we have $ g(f(u_1,v_1,0^{(m-2)}), \cdots,f(u_n,v_n,0^{(m-2})) \subseteq P \backslash \phi(P)$ by Theorem \ref{7}. Hence  $g(s,f(u_i,v_i,0^{(m-2)}),1_K^{(n-2)}) \subseteq P$ and so $f(g(s,u_i,1_K^{(n-2)}),g(s,v_i,1_K^{(n-2)}),0^{(m-2)}) \subseteq P$ 
which implies $g(s,u_i,1_K^{(n-2)}) \in P$ as $g(s,v_i,1_K^{(n-2)}) \in P$ or we conclude that $g(f(u_1,v_1,0^{(m-2)}),\cdots,\widehat{f(u_{i},v_{i},0^{(m-2)})},s,\cdots,f(u_n,v_n, 0^{(m-2)})) \subseteq \delta(P)$ which implies $g(u_1^{i-1},s,u_{i+1}^n) \in \delta(P)$ by Theorem \ref{7}. This contradicts the fact that $P$ is not $\delta$-$S$-primary hyperideal of $K$. Consequenly, $g(P^{(n)}) \subseteq \phi(P)$.
\end{proof} 
A proper hyperideal $P$ of $K$ may not be $n$-ary $\phi$-$\delta$-$S$-primary when it holds $g(P^{(n)}) \subseteq \phi(P)$. The following example verified this claim.
\begin{example} \label{madar}
Consider Krasner $(4,3)$-hyperring $(\mathbb{Z}_{5^{5r}},+,\cdot)$ where $r>4$ and also $+$ and $\cdot$ are usual addition and multiplication. Let $S=\{1\}$.  Then $P=\langle 5^r \rangle$ is not a $5$-ary $\phi_5$-$\delta_0$-$S$-primary  hyperideal of $\mathbb{Z}_{5^{5r}}$ since $5.5.5.5.5^{r-4} \in P \backslash \phi_5(P)$ but $5,5^{r-4} \notin P$.
\end{example}
%The following result is a version of Theorem 4.4 (1) and Theorem 4.10 (1) in \cite{davvazz}.
By using Theorem \ref{8}, we can deduce  the following corollaries. 
\begin{corollary}
Let $P$ be  a strongly  $n$-ary $\phi$-$\delta$-$S$-primary hyperideal of $K$ associated to $s \in S$ but is not an $n$-ary $\delta$-$S$-primary hyperideal. Then $rad(P)=rad(\phi(P))$.
\end{corollary}
\begin{proof}
Assume that $P$ is  a strongly  $n$-ary $\phi$-$\delta$-$S$-primary hyperideal of $K$ associated to $s \in S$ but is not an $n$-ary $\delta$-$S$-primary hyperideal. Then we get the result that $g(P^{(n)}) \subseteq \phi(P) \subseteq P$ by Theorem \ref{8}. Thus we conclude that $rad(P)=rad(\phi(P))$. 
\end{proof}
\begin{corollary}\label{9}
If $P$ is  a strongly  $n$-ary $\phi$-$\delta$-$S$-primary hyperideal of $K$ associated to $s \in S$ and $1_K \in S$, then $P \subseteq rad(\phi(P))$ or $g(s,rad(\phi(P)),1_K^{(n-2)}) \subseteq \delta(P)$.
\end{corollary}
\begin{proof}
Assume that $P \nsubseteq rad(\phi(P))$. This means that $g(P^n) \nsubseteq \phi(P)$. Since $P$ is  a strongly  $n$-ary $\phi$-$\delta$-$S$-primary hyperideal of $K$ associated to $s \in S$, we conclude that $P$ is an $n$-ary $\delta$-$S$-primary hyperideal of $K$ by Theorem \ref{8}. Now, let $g(s,rad(\phi(P)),1_K^{(n-2)}) \nsubseteq \delta(P)$. Then we have $g(s,u,1_K^{(n-2)}) \notin \delta(P)$ for some $u \in rad(\phi(P))$. Then 
there exists $r \in \mathbb {N}$ such that $g(u^ {(r)} , 1_K^{(n-r)} ) \in \phi(P) \subseteq P$ for $r \leq n$, or $g_{(x)} (u^ {(r)} ) \in \phi(P) \subseteq P$ for $r = x(n-1) + 1$. Assume that $r$ is a minimal integer satisfying the possibilities. In the first case, since $P$ is an $n$-ary $\delta$-$S$-primary hyperideal of $K$ associated to $s \in S$,  $g(u,g(u^{(r-1)},1_K^{(n-r+1)}),1_K^{(n-2)})=g(u^ {(r)} , 1_K^{(n-r)} ) \in  P$ and $g(s,u,1_K^{(n-2)}) \notin \delta(P)$, we get the result that $g(s,g(u^{(r-1)},1_K^{(n-r+1)}),1_K^{(n-2)}) \in P$. Again, since $g(u,g(s,u^{(r-2)},1_K^{(n-r+1)}),1_K^{(n-2)}) \in P$ and $g(s,u,1_K^{(n-2)}) \notin \delta(P)$, we obtain $g(s,g(s,u^{(r-2)},1_K^{(n-r+1)}),1_K^{(n-2)}) \in P$. By continuing the process, we get the result that $g(s^{(r-1)},u,1_K^{(n-r)})  \in P$. It follows that $g(s^{(r)},1_K^{(n-r)}) \in P$ as $g(s,u,1_K^{(n-2)}) \notin \delta(P)$. This contradicts the fact that $P \cap S=\varnothing$.  In the second case, by using a similar argument, we get a contradiction. Thus $g(s,rad(\phi(P)),1_K^{(n-2)}) \subseteq \delta(P)$.
\end{proof}
\begin{theorem}\label{10}
Assume that $P$ is a proper hyperideal of $K$. Then $P$ is a strongly $n$-ary $\phi$-$\delta$-$S$-primary hyperideal of $K$ associated  to $s \in S$ if and only if   $(P : u) = (\phi(P) : u)$ or $(P : u) \subseteq (P : s)$ for every $u \in K \backslash (\delta(P) : s )$.
\end{theorem}
\begin{proof}
$\Longrightarrow$  Let $P$ be  a strongly $n$-ary $\phi$-$\delta$-$S$-primary hyperideal of $K$ associated  to $s \in S$. Take any $u \in K \backslash (\delta(P) : s )$. So $g(s,u,1_K^{(n-2)})\notin \delta(P)$. Assume that $(P : u) \neq  (\phi(P) : u)$ and $ a \in (P : u)$. Let $a \notin  (\phi(P) : u)$. Then we conclude that  $g(s,a, 1_K^{(n-2)}) \in P$ as $g(u,a,1_K^{(n-2)}) \in P \backslash \phi(P)$ and $g(s,u,1_K^{(n-2)})\notin \delta(P)$. Therefore $a \in (P : s)$, as needed. Now, let $a \in  (\phi(P) : u)$. Take any $b \in (P : u) \backslash(\phi(P) : u)$. Then we have $g(b,u,1_K^{(n-2)}) \in P \backslash \phi(P)$. By the hypothesis, we obtain $g(s,b,1_K^{(n-2)}) \in P$ as $g(s,u,1_K^{(n-2)})\notin \delta(P)$. Since $g(b,u,1_K^{(n-2)}) \in P \backslash \phi(P)$, we get $g(u,f(a,b,0^{(m-2)}),1_K^{(n-2)})=f(g(u,a,1_K^{(n-2)}),g(u,b,1_K^{(n-2)}),0^{(m-2)}) \subseteq P \backslash \phi(P)$. Since $P$ is a strongly $n$-ary $\phi$-$\delta$-$S$-primary hyperideal of $K$ associated  to $s \in S$ and $g(s,u,1_K^{(n-2)})\notin \delta(P)$, we obtain $f(g(s,a,1_K^{(n-2)}),g(s,b,1_K^{(n-2)}),0^{(m-2)})=g(s,f(a,b,0^{(m-2)}),1_K^{(n-2)}) \subseteq P$. Then we get the result that  $g(s,a,1_K^{(n-2)}) \in P$ as $g(s,b,1_K^{(n-2)}) \in P$. It follows that $a \in (P : s)$ and so $(P : u) \subseteq (P : s)$.

$\Longleftarrow$ Assume that $g(P_1^n) \subseteq P \backslash \phi(P)$ for some hyperideals $P_1^n$ of $P$ but neither $g(s,P_i,1_K^{(n-2)}) \subseteq P$ nor $g(P_1^{i-1},s,P_{i+1}^n) \subseteq \delta(P)$ for all $i \in \{1,\cdots,n\}$. Since $g(P_1^{i-1},s,P_{i+1}^n) \nsubseteq \delta(P)$, there exists $u_j \in P_j$ for each $j \in \{1,\cdots,\widehat{i},\cdots,n\}$ such that $g(g(u_1^{i-1},1_K,u_{i+1}^n),s,1_K^{(n-2)})=g(u_1^{i-1},s,u_{i+1}^n) \notin \delta(P)$ and so $g(u_1^{i-1},1_K,u_{i+1}^n) \in g(P_1^{i-1},1_K,P_{i+1}^n) \backslash  (\delta(P) : s)$. By the hypothesis, we have $(P : g(u_1^{i-1},1_K,u_{i+1}^n)) = (\phi(P) : g(u_1^{i-1},1_K,u_{i+1}^n))$ or $(P : g(u_1^{i-1},1_K,u_{i+1}^n)) \subseteq (P : s)$. Since $P_i  \subseteq (P : g(u_1^{i-1},1_K,u_{i+1}^n)) \backslash (P : s)$, we conclude that $(P : g(u_1^{i-1},1_K,u_{i+1}^n)) = (\phi(P) : g(u_1^{i-1},1_K,u_{i+1}^n))$ which implies $g(u_1^{i-1},P_i,u_{i+1}^n) \subseteq \phi(P)$. Now, take any  $u \in g(P_1^{i-1},1_K,P_{i+1}^n) \cap (\delta(P) : s)$. Since  $f(g(u_1^{i-1},1_K,u_{i+1}^n),u,0^{(m-2)})$ is a subset of  $g(P_1^{i-1},1_K,P_{i+1}^n) \backslash (\delta(P) : s)$, we get $P_i \subseteq (P : f(g(u_1^{i-1},1_K,u_{i+1}^n),u,0^{(m-2)}))= (\phi(P) : f(g(u_1^{i-1},1_K,u_{i+1}^n),u,0^{(m-2)}))$. Therefore we get  $g(P_i,u,1_K^{(n-2)}) \subseteq \phi(P)$ since 
$f(g(u_1^{i-1},P_i,u_{i+1}^n),g(P_i,u,1_K^{(n-2)}),0^{(m-2)}) \subseteq \phi(P)$ and $g(u_1^{i-1},P_i,u_{i+1}^n) \subseteq \phi(P)$. This implies that $g(P_1^n) \subseteq \phi(P)$ which is impossible.  Thus $P$ is a strongly $n$-ary $\phi$-$\delta$-$S$-primary hyperideal of $K$.
\end{proof}
\begin{theorem}\label{11}
Let  $P$ be a proper hyperideal of $K$. Then $P$ is a strongly $n$-ary $\phi$-$\delta$-$S$-primary hyperideal of $K$ associated  to $s \in S$ if and only if   $(P : u) = (\phi(P) : u)$ or $(P : u) \subseteq (\delta(P) : s)$ for every $u \in K \backslash (P : s )$.
\end{theorem}
\begin{proof}
The proof can be easily obtained in a similar manner to Theorem \ref{10}.
\end{proof}
\begin{theorem}\label{12}
Assume that  $P$ is a strongly $n$-ary $\phi$-$\delta$-$S$-primary hyperideal of $K$ associated  to $s \in S$ and $Q$ is a hyperideal of $K$ with $Q \nsubseteq P$. If $(\delta(P) : Q) \subseteq \delta((P : Q))$ and $(\phi(P) : Q) \subseteq \phi((P : Q))$, then $(P : Q)$ is an $n$-ary $\phi$-$\delta$-$S$-primary hyperideal of $K$.
\end{theorem}
\begin{proof}
Let $g(u_1^n) \in (P : Q) \backslash \phi((P : Q))$ for $u_1^n \in K$. Since $(\phi(P) : Q) \subseteq \phi((P : Q))$, we get $g(u_i,g(u_1^{i-1},Q,u_{i+1}^n),1_K^{(n-2)})=g(g(u_1^n),Q,1_K^{(n-2)}) \subseteq P \backslash \phi(P)$. Therefore  $g(s,\langle u_i \rangle,1_K^{(n-2)}) \subseteq P$ or  $g(g(\langle u_1 \rangle, \cdots, \langle u_{i-1} \rangle,s, \langle u_{i+1} \rangle, \cdots, \langle u_n \rangle),Q,1_K^{(n-2)})=g(s,g(\langle u_1 \rangle, \cdots, \langle u_{i-1}\rangle,Q, \langle u_{i+1} \rangle, \cdots, \langle u_n \rangle),1_K^{(n-2)}) \subseteq \delta(P)$ as $P$ is  strongly $n$-ary $\phi$-$\delta$-$S$-primary  and  $g(\langle u_i \rangle, g(\langle u_1 \rangle, \cdots, \langle u_{i-1}\rangle,Q, \langle u_{i+1} \rangle, \cdots, \langle u_n \rangle),1_K^{(n-2)}) \subseteq P \backslash \phi(P)$. This implies that $g(s,u_i,1_K^{(n-2)}) \in P \subseteq (P : Q)$ or $g(u_1^{i-1},s,u_{i+1}^n) \in ( \delta(P) : Q) \subseteq \delta(P : Q)$. Consequently, $(P : Q)$ is an $n$-ary $\phi$-$\delta$-$S$-primary hyperideal of $K$.
\end{proof}
\begin{theorem} \label{13} 
Assume that  $P$ is a strongly  $n$-ary $\phi$-$\delta$-$S$-primary hyperideal of $K$ associated to $s \in S$ with $1_K \in S$ but it is not an $n$-ary $\delta$-$S$-primary hyperideal.  If $(P : s)=(\delta(P) : s)$, then  $g(g(s,rad(\phi(P)),1_K^{(n-2)}),P^{(n-1)}) \subseteq \phi(P)$.
\end{theorem}
\begin{proof}
Suppose that $P$ is a strongly $n$-ary $\phi$-$\delta$-$S$-primary hyperideal of $K$ associated to $s \in S$ but $P$ is not an $n$-ary $\delta$-$S$-primary hyperideal.  Take any $u \in rad(\phi(P))$. If $u \in (P : s)$, then $g(s,u,1_K^{(n-2)}) \in P$ and so $g(g(s,u,1_K^{(n-2)}),P^{(n-1)}) \subseteq g(P^{(n)}) \subseteq \phi(P)$ by Theorem \ref{8}. Now, let $u \notin (P : s)=(\delta(P) : s)$. This implies that $(P : u) \subseteq (P : s)=(\delta(P) : s)$ or $(P :u) =(\phi(P) : u)$ by Theorem \ref{10}. The first case leads to the following contradiction. Since $u \in rad(\phi(P))$, there exists $r \in \mathbb{N}$ such that $g(u^{(r)},1_K^{(n-r)}) \in \phi(P)$ for $r \leq n$ or $g_{(x)}(u^{(r)})\in \phi(P)$ for $r=x(n-1)+1$. Assume that $r$ is a minimal integer satisfying the possibilities. If $g(u^{(r)},1_K^{(n-r)}) \in \phi(P)$ for $r \leq n$, then $g(u^{(r-1)},1_k^{(n-r+1)}) \in (P : u) \subseteq (P : s)$ which means $g(g(u^{(r-1)},1_K^{(n-r+1)}),s,1_K^{(n-2)})=g(u^{(r-1)},s,1_K^{(n-r)}) \in P$. If $g(u^{(r-1)},s,1_K^{(n-r)})= g(g(s,u,1_K^{(n-2)}),u^{(r-2)},1_K^{(n-r+1)}) \notin \phi(P)$, then we get $g(s,g(s,u,1_K^{(n-2)}),1_K^{(n-2)})=g(s^2,u,1_K^{(n-3)}) \in P$ or $g(s,u^{(r-2)},1_K^{(n-r+1)}) \in \delta(P)$. From $g(s^2,u,1_K^{(n-3)}) \in P$ it follows that $g(s^{(2)},1_K^{(n-2)}) \in (P : u) \subseteq (P :s)$ which means $g(s^{(3)},1_K^{(n-3)}) \in P$ which is impossible. On the other hand, by $g(s,u^{(r-2)},1_K^{(n-r+1)}) \in \delta(P)$ we conclude that  $g(u^{(r-2)}),1_K^{(n-r+2)}) \in (\delta(P) : s)=(P : s)$. Therefore we get the result that $g(s,g(u^{(r-2)},1_K^{(n-r+2)}),1_K^{(n-2)})=g(s,u^{(r-2)},1_K^{(n-r+1)}) \in P \backslash \phi(P)$ as $g(s,u^{(r-1)},1_K^{(n-r)}) \notin \phi(P)$. By continuing the process, we have $g(s,u,1_K^{(n-2)}) \in P$ which is impossible. Then we have $g(u^{(r-1)},s,1_K^{(n-t)}) \in \phi(P)$. Assume that $t$ is a minimal integer satisfying $g(s,u^{(t)},1_K^{(n-t-1)}) \in \phi(P)$. Let $g(g(s,u,1_K^{(n-2)}),u_1^{n-1}) \notin \phi(P)$ for some $u_1^{n-1} \in P$. So $ f(g(s,u^{(t)},1_K^{(n-t-1)}),g(g(s,u,1_K^{(n-2)}),u_1^{n-1}),0^{(m-2)}) \subseteq P \backslash \phi(P)$. Then $g(g(s,u,1_K^{(n-2)}),f(u^{(t-1)},g(u_1^{n-1},1_K),0^{(m-2)}),1_A^{(n-2)}) \subseteq P \backslash \phi(P)$. Since   $P$ is a strongly $n$-ary $\phi$-$\delta$-$S$-primary hyperideal of $K$ 
 and $g(s^2,u,1_K^{(n-3)})=g(s,g(s,u,1_K^{(n-2)}),1_K^{(n-2)}) \notin P$,  $g(s,f(u^{(t-1)},g(u_1^{n-1},1_K),0^{(m-2)}),1_K^{(n-2)}) \subseteq \delta(P)$ which means $f(u^{(t-1)},g(u_1^{n-1},1_K),0^{(m-2)}) \subseteq (\delta(P) : s)=(P : s)$ which implies $g(s,f(u^{(t-1)},g(u_1^{n-1},1_K),0^{(m-2)}),1_K^{(n-2)}) \subseteq P$. Since $g(s,u_1^{n-1}) \in P$, we have $ g(s,u^{(t-1)},1_K^{(n-t)}) \in P \backslash \phi(P)$ which implies $g(s,u,1_K^{(n-2)}) \in P$ which is a contradiction. If $g_{(x)}(u^{(r)}) \in \phi(P)$ for $r=x(n-1)+1$, then by using a similar argument, we get a contradiction. In the second case, we get the result that $g(u,P,1_K^{(n-2)}) \subseteq \phi(P)$ as $P \subseteq (P : u)$. Hence $g(s,g(u,P,1_K^{(n-2)}),P^{(n-2)})=g(g(s,u,1_K^{(n-2)}),P^{(n-1)})\in \phi(P)$. Thus we have $g(g(s,rad(\phi(P)),1_K^{(n-2)}),P^{(n-1)}) \subseteq \phi(P)$.
\end{proof}
As a result of   Theorem \ref{8} and Theorem \ref{13}, we give the following corollariy.
\begin{corollary} \label{14}
Assume that $P$ and $Q$ are strongly  $n$-ary $\phi$-$\delta$-$S$-primary hyperideals of $K$ associated to $s \in S$ with $1_K \in S$ but they are not  $n$-ary $\delta$-$S$-primary hyperideals. If $(P : s)=(\delta(P) : s)$, $(Q: s)=(\delta(Q) : s)$ and $\phi(Q) \subseteq \phi(P)$, then  $g(g(s,Q,1_K^{(n-2)}),P^{(n-1)})) \subseteq \phi(P).$ 
\end{corollary}
\begin{proof}
Suppose that $P$ and $Q$ are strongly  $n$-ary $\phi$-$\delta$-$S$-primary hyperideals of $K$ associated to $s \in S$ with $1_K \in S$ but they are not  $n$-ary $\delta$-$S$-primary hyperideals. Then we get the result that  $Q \subseteq rad(\phi(Q)) \subseteq rad(\phi(P))$ by Theorem \ref{8}. Therefore  $g(g(s,Q,1_K^{(n-2)}),P^{(n-1)})=g(g(s,rad(\phi(P)),1_K^{(n-2)}),P^{(n-1)}) \subseteq \phi(P) $ by Theorem \ref{13}. 
\end{proof}
Let $\phi_S(S^{-1}I)=S^{-1}\phi(I)$ and $\delta_S(S^{-1}I)=S^{-1}\delta(I)$ for every hyperideal $I$ of $K$. Assume that $P$ is a proper hyperideal of $K$ satisfying $\phi(P : u)=(\phi(P) : u)$ and $\delta(P : u)=(\delta(P) : u)$ for every $u \in K$ and suppose that $\delta_S(S^{-1} P) \neq S^{-1}K$ if $S^{-1}P \neq S^{-1}K$. Furthermore, let $\delta(S^{-1} P \cap K) =S^{-1}\delta(P) \cap K$. Then the following theorem holds under these mentioned conditions.
\begin{theorem}\label{15}
Let  $P$ be a proper hyperideal of $K$ such that $P \cap S = \varnothing$, $\phi(P)=(\phi(P) : s)$ for some $s \in S$ and $(\phi(P) : t) \subseteq (\phi(P) : s)$ for every $t \in S$. Then the
following assertions are equivalent:
\begin{itemize}
\item[\rm{(i)}]~ $P$ is an $n$-ary $\phi$-$\delta$-$S$-primary hyperideal of $K$ associated to $s \in S$.
\item[\rm{(ii)}]~ $(P : s)$ is an $n$-ary $\phi$-$\delta$-primary hyperideal of $K$.
\item[\rm{(iii)}]~$S^{-1}P$ is an $n$-ary $\phi_S$-$\delta_S$-primary hyperideal of $S^{-1}K$ and $(P : t) \subseteq (P : s)$ for every $t \in S$. 
 \item[\rm{(iv)}]~ $S^{-1}P$ is an $n$-ary $\phi_S$-$\delta_S$-primary hyperideal of $S^{-1}K$ and $S^{-1}P \cap K =(P : s)$. 
\end{itemize}
\end{theorem}
\begin{proof}
(i) $\Longrightarrow$ (ii) It is clear that $(P : s) \neq K$. Let $g(u_1^n) \in (P : s) \backslash \phi(P : s)$ for $u_1^n \in K$. This implies that $g(g(s,u_i,1_K^{(n-2)}),g(u_1^{i-1},1_K,u_{i+1}^n),1_K^{(n-2)})=g(s,g(u_1^n),1_K^{(n-2)}) \in P \backslash \phi(P)$. Since $P$ is an $n$-ary $\phi$-$\delta$-$S$-primary hyperideal of $K$ associated to $s \in S$, we get $g(s^2,u_i,1_K^{(n-2)})=g(s,g(s,u_i,1_K^{(n-2)}),1_K^{(n-2)}) \in P \backslash \phi(P)$ or $g(s,g(u_1^{i-1},1_K,u_{i+1}^n),1_K^{(n-2)}) \in \delta(P)$. Therefore we have $g(s,u_i,1_K^{(n-2)}) \in P$ or $g(s,g(u_1^{i-1},1_K,u_{i+1}^n),1_K^{(n-2)}) \in \delta(P)$. Then we conclude that $u_i \in (P : s)$ or $g(u_1^{i-1},1_K,u_{i+1}^n) \in (\delta(P) : s)=\delta( P : s)$. Thus, $(P : s)$ is an $n$-ary $\phi$-$\delta$-primary hyperideal of $K$.

(ii) $\Longrightarrow$ (iii) Let $G(\frac{u_1}{s_1},\cdots,\frac{u_n}{s_n}) \in S^{-1}P \backslash \phi_S(S^{-1}P)$ for $\frac{u_1}{s_1},\cdots,\frac{u_n}{s_n} \in S^{-1}K$. This means that $\frac{g(u_1^n)}{g(s_1^n)} \in S^{-1}P \backslash \phi_S(S^{-1}P)$ and so $g(t,g(u_1^n),1_K^{(n-2)}) \in P$ for some $t \in S$. If $g(t,g(u_1^n),1_K^{(n-2)}) \in \phi(P)$, then $G(\frac{u_1}{s_1},\cdots,\frac{u_n}{s_n})=\frac{g(t,g(u_1^n),1_K^{(n-2)})}{g(t,g(s_1^n),1_K^{(n-2)})} \in S^{-1}\phi(P)=\phi_S(S^{-1}P)$ which is impossible. Then $g(t,g(u_1^n),1_K^{(n-2)}) \in P \backslash \phi(P)$ which means $g(u_i,g(u_1^{i-1},t,u_{i+1}^n),1_K^{(n-2)})=g(t,g(u_1^n),1_K^{(n-2)}) \in (P : s) \backslash (\phi(P) : s)$. By the hypothesis, we get the result that $u_i \in (P : s)$ or $g(u_1^{i-1},t,u_{i+1}^n) \in \delta(P : s)=(\delta(P) : s)$. Therefore we have $g(s,u_i,1_K^{(n-2)}) \in P$ which means $\frac{u_i}{s_i}=\frac{g(s,u_i,1_K^{(n-2)})}{g(s,s_i,1_K^{(n-2)})} \in S^{-1}P$ or $g(s,t,g(u_1^{i-1},1_K,u_{i+1}^n),1_K^{(n-3)}) \in \delta(P)$ which implies $G(\frac{u_1}{s_1},\cdots,\widehat{\frac{u_i}{s_i}},\cdots,\frac{u_n}{s_n})=\frac{g(s,t,g(u_1^{i-1},1_K,u_{i+1}^n),1_K^{(n-3)})}{g(s,t,g(s_1^{i-1},1_K,s_{i+1}^n),1_K^{(n-3)})} \in S^{-1}\delta(P)=\delta_S(S^{-1}P)$. This shows that $S^{-1}P$ is an $n$-ary $\phi_S$-$\delta_S$-primary hyperideal of $S^{-1}K$. Now, let $t \in S$ and $u \in ( P : t)$. If $u \in (\phi(P) : t)$, then we get $u \in (\phi(P) : t) \subseteq (P : s)$, as needed. Let  $u \notin (\phi(P) : t)$. Then we conclude that $g(t,u,1_A^{(n-2)}) \in ( P : s) \backslash  \phi(P : s)$ as $g(t,u,1_K^{(n-2)}) \in P \backslash \phi(P)$. Since $t \notin \delta(P : s)$, we obtain $u \in (P : s)$. Thus $(P : t) \subseteq (P : s)$.

(iii) $\Longrightarrow$ (iv) Assume that $s \in S$ with $(P : t) \subseteq (P : s)$ for every $t \in S$. We know that $(P : s) \subseteq S^{-1}P \cap K$. Take any $u \in S^{-1}P \cap K$. Since $\frac{u}{1_K}=\frac{v}{w}$ for some $v \in P$ and $w \in S$, we get $g(t,u,1_K^{(n-2)}) \in P$ for some $t \in S$. Hence, we obtain $u \in (P :t)$ and so $u \in (P : s)$. This means $S^{-1}P \cap K \subseteq (P : s)$. Consequently, $S^{-1}P \cap K = (P : s)$.

(iv) $\Longrightarrow$ (i) Let $S^{-1}P$ be an $n$-ary $\phi_S$-$\delta_S$-primary hyperideal of $S^{-1}K$ and $S^{-1}P \cap K =(P : s)$. Assume that $g(u_1^n) \in P\backslash \phi(P)$ for $u_1^n \in K$. Then we have $G(\frac{u_1}{1_K},\cdots,\frac{u_n}{1_K}) \in S^{-1}P \backslash \phi_S(S^{-1}P)$. By the hypothesis, we conclude that $\frac{u_i}{1_K} \in S^{-1}P$ or $G(\frac{u_1}{1_K},\cdots,\widehat{\frac{u_i}{1_K}},\cdots\frac{u_n}{1_K}) \in \delta_S(S^{-1}P)=S^{-1}\delta(P)$ for some $i \in \{1,\cdots,n\}$. In the first case, we get $g(t,u_i,1_K^{(n-2)}) \in P$ for some $t \in S$ and so $u_i=\frac{g(t,u_i,1_K^{(n-2)})}{g(t,1_K^{(n-1)})} \in S^{-1}P \cap K = (P : s)$. It follows that $g(s,u_i,1_K^{(n-2)}) \in P$, as required. In the second case, we have $g(u_1^{i-1},t^{\prime},u_{i+1}^n) \in \delta(P)$ for some $t^{\prime} \in S$ and so $g(u_1^{i-1},1_K,u_{i+1}^n)=\frac{g(t^{\prime},g(u_1^{i-1},1_K,u_{i+1}^n),1_K^{(n-2)})}{g(t^{\prime},1_K^{(n-1)})} \in S^{-1}\delta(P) \cap K$ which implies $g(u_1^{i-1},1_K,u_{i+1}^n) \in (\delta(P) :s)$ by the assumption. Therefore, $g(u_1^{i-1},s,u_{i+1}^n) \in \delta(P)$. Thus, we conclude that $P$ is an $n$-ary $\phi$-$\delta$-$S$-primary hyperideal of $K$ associated to $s \in S$.
\end{proof}
Assume that $(K_1, f_1, g_1)$ and $(K_2, f_2, g_2)$ are two Krasner $(m, n)$-hyperrings. Recall from \cite{d1} that a mapping 
$k : K_1 \longrightarrow K_2$ is called a homomorphism if for all $u^m _1,v^n_ 1 \in K_1$, we have
$ k(f_1(u_1^m)) = f_2(k(u_1),\cdots,k(u_m))$, and 
$ k(g_1(v_1,\cdots, v_n)) = g_2(k(v_1),\cdots,h(v_n))$.  Suppose that $\phi$ and $\psi$ are  reduction functions of hyperideals of $K_1$ and $K_2$, respectively, and also  $\delta$ and $\gamma$ are  expansion functions of hyperideals of $K_1$ and $K_2$, respectively.   Then the homomorphism $k$ is called a $(\delta,\phi)$-$(\gamma,\psi)$-homomorphism if $\phi(k^{-1}(P_2))=k^{-1}(\psi(P_2))$ and $\delta(k^{-1}(P_2))=k^{-1}(\gamma(P_2))$ for each hyperideal $P_2$ of $K_2$. Let $k$ be a $(\delta,\phi)$-$(\gamma,\psi)$-epimorphism from $K_1$ to $K_2$ and let $P_1$ be a hyperideal of $K_1$ with $Ker (k) \subseteq P_1$. It is easy to see that $\psi(k(P_1))=k(\phi(P_1))$ and $\gamma(k(P_1))=k(\delta(P_1))$.
\begin{theorem} \label{homo1}
Let $k: K_1 \longrightarrow K_2$ be a nonzero $(\delta,\phi)$-$(\gamma,\psi)$-epimorphism and $S$ be an $n$-ary multiplicative subset of $K_1$ with $1_{K_1} \in S$. If $P_2$ is an $n$-ary  $\psi$-$\gamma$-$k(S)$-primary hyperideal of $K_2$ associated to $f(s) \in f(S)$, then $k^{-1}(P_2)$ is an $n$-ary $\phi$-$\delta$-$S$-primary hyperideal of $K_1$ associated to $s \in S$. 
\end{theorem}
\begin{proof}
Let $S$ be an $n$-ary multiplicative subset of $K_1$ with $1_{K_1} \in S$. Since $k$ is a nonzero epimorphism, we conclude that $k(S)$ is an $n$-ary multiplicative subset of $K_2$ with $1_{K_2}=k(1_{K_1}) \in k(S)$.  Let  $P_2$ be an $n$-ary  $\psi$-$\gamma$-$k(S)$-primary hyperideal of $K_2$ associated to $f(s) \in f(S)$.   Assume that $ g_1(u_1^n) \in k^{-1}(P_2) \backslash \phi(k^{-1}(P_2) )$ for $u_1^n \in K_1$. Then $ k(g_1(u_1^n))=g_2(k(u_1),\cdots,k(u_n)) \in P_2 \backslash \psi(P_2)$. Therefore we get  $g_2(k(s),k(u_i),1_{K_2}^{(n-2)})=k(g_1(s,u_i,1_{K_1}^{(n-2)})) \in P_2$  for some $i \in \{1,\cdots,n\}$ or $g_2(k(u_1),\cdots,k(u_{i-1}),k(s),k(u_{i+1}),\cdots,k(u_n))=k(g_1(u_1^{i-1},s,u_{i+1}^n)) \in \gamma(P_2)$.  So we conclude that   $g_1(s,u_i,1_{K_1}^{(n-2)}) \in k^{-1}(P_2)$ or  $g_1(u_1^{i-1},s,u_{i+1}^n) \in k^{-1}(\gamma(P_2))=\delta(k^{-1}(P_2))$. Thus $k^{-1}(P_2)$ is an $n$-ary $\phi$-$\delta$-$S$-primary hyperideal of $K_1$ associated to $s \in S$. 
%(ii) Suppose that $y \in h(Q) \cap h(S)$. So we have $y=h(x)=h(s)$ for some $x \in Q$ and $s \in S$. Then $f(h(x),-h(s),0^{(m-2)})=h(f(x,-s,0^{(m-2)})$
\end{proof}
\section{$\phi$-$\delta$-$S$-primary hyperideals  in direct product of commutative Krasner $(m,n)$-hyperrings}
Let  $(K_1, f_1, g_1)$ and $(K_2, f_2, g_2)$ be two commutative Krasner $(m,n)$-hyperrings and, $1_{K_1}$ and $1_{K_2}$ be  scalar identities of $K_1$ and $K_2$, respectively. Then   the direct product of  $K_2$ and $K_2$,  denoted by $K_1 \times K_2$,  is a commutative Krasner $(m, n)$-hyperring  with  the $m$-ary hyperoperation
and $n$-ary operation  defined  by

$\hspace{1cm} f_1 \times f_2((u_{1}, v_{1}),\cdots,(u_m,v_m)) = \{(u,v) \ \vert \ \ u \in f_1(u_1^m), v \in f_2(v_1^m) \}$

$\hspace{1cm} g_1 \times g_2 ((w_1,z_1),\cdots,(w_n,z_n)) =(g_1(w_1^n),g_2(z_1^n)) $,\\
for  $u_1^m,w_1^n \in K_1$ and $v_1^m,z_1^n \in K_2$ \cite{rev2}. 
Suppose that  $S_1$ and $S_2$ are $n$-ary multiplicative subsets of $K_1$ and $K_2$, respectively. Then $S_1 \times S_2$ is $n$-ary multiplicative subset of $K_1 \times K_2$. Assume that $\phi_1$ and $\phi_2$ are  reduction functions of hyperideals of $K_1$ and $K_2$, respectively, and also  $\delta_1$ and $\delta_2$ are  expansion functions of hyperideals of $K_1$ and $K_2$, respectively. Then $\hat{\phi}$ and $\hat{\delta}$ defined as $\hat{\phi}(P_1 \times P_2)=\phi_1(P_1) \times \phi_2(P_2)$ and $\hat{\delta}(P_1 \times P_2)=\delta_1(P_1) \times \delta_2(P_2)$ are reduction and expansion functions of hyperideals $K_1 \times K_2$.
\begin{theorem}\label{16}
Let $S_1 $ and $ S_2$ be $n$-ary multiplicative subsets of $K_1$ and $K_2$, respectively, and $P_1$ be a hyperideal of $K_1$. Assume that $\phi_1$ and $\phi_2$ are  reduction functions of hyperideals of $K_1$ and $K_2$, respectively, and also  $\delta_1$ and $\delta_2$ are  expansion functions of hyperideals of $K_1$ and $K_2$, respectively,  such that $\phi_2(K_2) \neq K_2$. Then $P_1 \times K_2$ is an $n$-ary $\hat{\phi}$-$\hat{\delta}$-$S_1 \times S_2$-primary hyperideal of $K_1 \times K_2$ associated to $(s_1,s_2) \in S_1 \times S_2$ if and only if $P_1$ is an $n$-ary $\delta_1$-$S_1$-primary hyperideal of $K_1$ associated to $s_1$ and $P_1 \times K_2$ is an $n$-ary $\hat{\delta}$-$S_1 \times S_2$-primary hyperideal of $K_1 \times K_2$ associated to $(s_1,s_2)$.
\end{theorem}
\begin{proof}
$\Longrightarrow$ Let $P_1 \times K_2$ be an $n$-ary $\hat{\phi}$-$\hat{\delta}$-$S_1 \times S_2$-primary hyperideal of $K_1 \times K_2$ associated to $(s_1,s_2) \in S_1 \times S_2$. Assume that $g(u_1^n) \in P_1$ for $u_1^n \in K_1$. Therefore we have $g_1 \times g_2((u_1,1_K),\cdots,(u_n,1_K))=(g(u_1^n),1_K) \in P_1 \times K_2 \backslash \hat{\phi}(P_1 \times K_2)$. By the hypothesis, we conclude that $g_1 \times g_2((s_1,s_2),(u_i,1_K),(1_{K_1},1_{K_2})^{(n-2)}) \in P_1 \times K_2$   or  $g_1\times g_2((s_1,s_2),(u_1,1_K),\cdots,\widehat{(u_i,1_K)},\cdots,(u_n,1_K)) \in \hat{\delta}(P_1 \times K_2)$ for some $i \in \{1,\cdots,n\}$. Then we have $(g_1(s_1,u_i,1_{K_1}^{(n-2)}),g_2(s_2,1_{K_2}^{(n-1)})) \in P_1 \times K_2$ which means $g_1(s_1,u_i,1_{K_1}^{(n-2)}) \in P_1$ or $(g_1(u_1^{i-1},s_1,u_{i+1}^n),g_2(s_2,1_{K_2}^{(n-1)})) \in \hat{\delta}(P_1 \times K_2)$ which implies $g_1(u_1^{i-1},s_1,u_{i+1}^n) \in \delta(P_1)$. This shows that $P_1$ is an $n$-ary $\delta_1$-$S_1$-primary hyperideal of $K_1$ associated to $s_1$. Now, assume that  $P_1 \times K_2$ is not an $n$-ary $\hat{\delta}$-$S_1 \times S_2$-primary hyperideal of $K_1 \times K_2$. Then we get the result that $g_1 \times g_2((P_1 \times K_2)^{(n)}) \subseteq \hat{\phi}(P_1 \times K_2)$ by Theorem \ref{8}. So, we conclude that $\phi_2(K_2)=K_2$ which is impossible. Hence,  $P_1 \times K_2$ is an $n$-ary $\hat{\delta}$-$S_1 \times S_2$-primary hyperideal of $K_1 \times K_2$ associated to $(s_1,s_2)$.

$\Longleftarrow$ It is obvious.
\end{proof}
\begin{theorem}\label{17}
 Let $\phi_1$ and $\phi_2$ be  reduction functions of hyperideals of $K_1$ and $K_2$, respectively, and let   $\delta_1$ and $\delta_2$ be expansion functions of hyperideals of $K_1$ and $K_2$, respectively. Suppose that $S_1 $ and $ S_2$ are $n$-ary multiplicative subsets of $K_1$ and $K_2$, respectively, and $P_1$ is a hyperideal of $K_1$.Then $P_1 \times K_2$ is an $n$-ary $\hat{\phi}$-$\hat{\delta}$-$S_1 \times S_2$-primary hyperideal  of $K_1 \times K_2$ associated to $(s_1,s_2) \in S_1 \times S_2$ that is not $\hat{\delta}$-$S_1 \times S_2$-primary if and only if $\phi_2(K_2)=K_2$, $\hat{\phi}(P_1 \times K_2) \neq \varnothing$, and $P_1$ is an $n$-ary $\phi_1$-$\delta_1$-$S_1$-primary hyperideal of $K_1$ associated to $s_1$ that is not $\delta_1$-$S_1$-primary.
\end{theorem}
\begin{proof}
Let $P_1 \times K_2$ be an $n$-ary $\hat{\phi}$-$\hat{\delta}$-$S_1 \times S_2$-primary hyperideal  of $K_1 \times K_2$ associated to $(s_1,s_2) \in S_1 \times S_2$ that is not $\hat{\delta}$-$S_1 \times S_2$-primary. Then we conclude that $g_1 \times g_2((P_1 \times K_2)^n) \subseteq \hat{\phi}(P_1 \times K_2)$ by Theorem \ref{8}, so $\hat{\phi}(P_1 \times K_2) \neq \varnothing$. Assume that $\phi_2(K_2) \neq K_2$. Then we get the result that $P_1 \times K_2$ is an $n$-ary $\hat{\delta}$-$S_1 \times S_2$-primary hyperideal of $K_1 \times K_2$ associated to $(s_1,s_2)$ by Theorem \ref{16}, which is impossible. Therefore  $\phi_2(K_2) =K_2$. Since $P_1 \times K_2$ is an $n$-ary $\hat{\phi}$-$\hat{\delta}$-$S_1 \times S_2$-primary hyperideal  of $K_1 \times K_2$ associated to $(s_1,s_2)$, $P_1$ is an $n$-ary $\phi_1$-$\delta_1$-$S_1$-primary hyperideal of $K_1$ associated to $s_1$. Now, assume that  $P_1$ is an $n$-ary $\delta_1$-$S_1$-primary hyperideal of $K_1$ associated to $s_1$. Then $P_1 \times K_2$ is an $n$-ary $\hat{\delta}$-$S_1 \times S_2$-primary hyperideal of $K_1 \times K_2$ associated to $(s_1,s_2)$ by Theorem \ref{homo1}, a contradiction. Conseuently, $P_1$ is an $n$-ary $\phi_1$-$\delta_1$-$S_1$-primary hyperideal of $K_1$ associated to $s_1$ that is not $\delta_1$-$S_1$-primary.

$\Longleftarrow$ Let $g_1 \times g_2((u_1,v_1),\cdots,(u_n,v_n))=(g_1(u_1^n),g_2(v_1^n)) \in P_1 \times K_2 \backslash \hat{\phi}(P_1 \times K_2)$ for $(u_1,v_1),\cdots,(u_n,v_n) \in K_1 \times K_2$. Since $\phi_2(K_2)=K_2$, we have $g(u_1^n) \in P_1 \backslash \phi_1(P_1)$. 
By the hypothesis, we get $g(s_1,u_i,1_{K_1}^{(n-2)}) \in P_1$ or $g(u_1^{i-1},s_1,u_{i+1}^n) \in \delta_1(P_1)$. Then we conclude  that $(g_1(s_1,u_i,1_{K_1}^{(n-2)}),g_2(s_2,v_i,1_{K_2}^{(n-2)})) \in P_1 \times K_2$ or  $(g_1(u_1^{i-1},s_1,u_{i+1}^n),g_2(u_1^{i-1},s_2,u_{i+1}^n)) \in \delta_1(P_1) \times \delta_2(K_2)$. In the first case, we have 
$g_1 \times g_2((s_1,s_2),(u_i,v_i),(1_{K_1},1_{K_2})^{(n-2)}) \in P_1 \times K_2$. In the second case, we get  $g_1 \times g_2((u_1,v_1),\cdots,(u_{i-1},v_{i-1}),(s_1,s_2),(u_{i+1},v_{i+1}),\cdots,(u_n,v_n)) \in \hat{\delta}(P_1 \times K_2)$. This means that $P_1 \times K_2$ is an $n$-ary $\hat{\phi}$-$\hat{\delta}$-$S_1 \times S_2$-primary hyperideal  of $K_1 \times K_2$ associated to $(s_1,s_2) \in S_1 \times S_2$. Now, suppose that $P_1 \times K_2$ is  an $n$-ary $\hat{\delta}$-$S_1 \times S_2$-primary hyperideal of $K_1 \times K_2$ associated to $(s_1,s_2)$. This implies that  $P_1$ is an $n$-ary  $\delta_1$-$S_1$-primary hyperideal of $K_1$ associated to $s_1$  which is impossible. Thus, $P_1 \times K_2$ is an $n$-ary $\hat{\phi}$-$\hat{\delta}$-$S_1 \times S_2$-primary hyperideal  of $K_1 \times K_2$ associated to $(s_1,s_2)$ that is not $\hat{\delta}$-$S_1 \times S_2$-primary.
\end{proof}
 \begin{theorem} \label{cart1}
Let $S_1 $ and $ S_2$ be $n$-ary multiplicative subsets of $K_1$ and $K_2$, respectively. Assume that $\phi_j$ is a reduction function of hyperideals of $K_j$ and $\delta_j$ is an expansion  function of hyperideal of $K_j$ for each $j \in \{1,2\}$. 
Suppose that  $P_j$ is a  hyperideals of $K_j$ satisfying $\phi_j (P_j) \neq P_j$  for each $j \in \{1,2\}$ such that if $P_j \neq K_j$, then $\delta_j(P_j) \neq K_j$. Then the following assertions  are equivalent:
 \begin{itemize} 
\item[\rm{(i)}]~ $P_1 \times P_2$ is an $n$-ary $\hat{\phi}$-$\hat{\delta}$-$S_1 \times S_2$-primary hyperideal of $K_1 \times K_2$ associated to $(s_1,s_2) \in S_1 \times S_2$.
\item[\rm{(ii)}]~   $P_2$ is an $n$-ary $\delta_2$-$S_2$-primary hyperideal of $K_2$  associated to $s_2 $ and $P_1=K_1$ or $P_2$ is an $n$-ary $\delta_2$-$S_2$-primary hyperideal of $K_2$  associated to $s_2 $ and $s_1 \in P_1 \cap S_1$ or $P_1$ is an $n$-ary $\delta_1$-$S_1$-primary hyperideal of $K_1$  associated to $s_1 $ and $P_2=K_2$ or $P_1$ is an $n$-ary $\delta_1$-$S_1$-primary hyperideal of $K_1$  associated to $s_1$ and $s_2 \in P_2 \cap S_2$.
\item[\rm{(iii)}]~ $P_1 \times P_2$ is an $n$-ary $\hat{\delta}$-$S_1 \times S_2$-primary hyperideal of $K_1 \times K_2$ associated to $(s_1,s_2) \in S_1 \times S_2$.
\end{itemize} 
\end{theorem}
\begin{proof}
(i) $\Longrightarrow$ (ii) 
If $P_j=K_j$ for $j \in \{1,2\}$ , then $P_j$ is an $n$-ary $\delta_j$-$S_j$-primary hyperideal of $K_j$  associated to $s_j$, by Theorem \ref{16}. Let $P_j \neq K_j$ for each $j \in \{1,2\}$. Let $u \in P_1$ and $v \in P_2 \backslash \phi_2(P_2)$. Then $g_1 \times g_2((u,1_{K_2}),(1_{K_1},1_{K_2})^{(n-2)},(1_{K_1},v))=(g_1(u,1_{K_2}^{(n-1)}),g_2(1_{K_1}^{(n-1)},v))
\in P_1 \times P_2 \backslash \hat{\phi}(P_1 \times P_2)$. Since $P_1 \times P_2$ is an $n$-ary $\hat{\phi}$-$\hat{\delta}$-$S_1 \times S_2$-primary hyperideal of $K_1 \times K_2$ associated to $(s_1,s_2)$, we obtain $g_1 \times g_2((s_1,s_2),(u,1_{K_2}),(1_{K_1},1_{K_2})^{(n-2)})=(g_1(s_1,u,1_{K_1}^{(n-2)}),g_2(s_2,1_{K_2}^{(n-1)})) \in P_1 \times P_2$  or $g_1 \times g_2((s_1,s_2),(1_{K_1},v),(1_{K_1},1_{K_2})^{(n-2)})=(g_1(s_1,1_{K_1}^{(n-2)}),g_2(s_2,v,1_{K_2}^{(n-2)})) \in \hat{\delta}(P_1 \times P_2)=\delta_1(P_1) \times \delta_2(P_2)$.
Then we get the result that  $s_2 \in P_2 \cap S_2$ or $s_1 \in \delta_1(P_1) \cap S_1=P_1 \cap S_1$. In the first case, we conclude that $P_1 \cap S_1=\varnothing$ as $P_1 \times P_2 \cap S_1 \times S_2 =\varnothing$. Assume that $g_1(u_1^n) \in P_1$ for $u_1^n \in K_1$. Since $s_2 \in P_2 \cap S_2$ and $s_2 \notin \phi_2(P_2)$, we have $g_1 \times g_2((u_1,s_2),(u_2,1_{K_2}),\cdots(u_n,1_{K_n}))=(g_1(u_1^n),g(s_2,1_{K_2}^{(n-2)})) \in P_1 \times P_2 \backslash \hat{\phi}(P_1 \times P_2)$.  Since $P_1 \times P_2$ is an $n$-ary $\hat{\phi}$-$\hat{\delta}$-$S_1 \times S_2$-primary hyperideal of $K_1 \times K_2$ associated to $(s_1,s_2)$,  we conclude that $g_1 \times g_2((s_1,s_2)(u_1,s_2),(1_{K_1},1_{K_2})^{(n-2)}) \in P_1 \times P_2$  or $g_1 \times g_2((s_1,s_2)(u_2,1_{K_2}),\cdots,(u_n,1_{K_2}) \in \hat{\delta}(P_1 \times P_2)$   or  for some $i \in \{2,\cdots,n\}$ we get $g_1 \times g_2((s_1,s_2)(u_i,1_{K_2}),(1_{K_1},1_{K_2})^{(n-2)}) \in P_1 \times P_2$ or $g_1 \times g_2((s_1,s_2)(u_1,s_2),(u_2,1_{K_2}),\cdots,\widehat{(u_i,1_{K_2})},\cdots,(u_n,1_{K_2})) \in \hat{\delta}(P_1 \times P_2)$. It follows that $(g_1(s_1,u_1,1_{K_1}^{(n-2)}),g_2(s_2^{(2)},1_{K_2}^{(n-2)}))  \in P_1 \times P_2$ or $(g_1(s_1,u_2^n),g_2(s_2,1_{K_2}^{(n-1)}) \in \hat{\delta}(P_1 \times P_2)$ or $(g_1(s_1,u_i,1_{K_1}^{(n-2)}),g_2(s_2,1_{K_2}^{(n-2)}) \in P_1 \times P_2$ for some $i \in \{2,\cdots,n\}$ or $(g_1(u_1^{i-1},s_1,u_{i+1}^n),g_2(s_2^{(2)},1_{K_2}^{(n-2)}) \in \hat{\delta}(P_1 \times P_2)$. Then we get the result that $g_1(s_1,u_i,1_{K_1}^{(n-2)}) \in P_1$ or $g_2(u_1^{i-1},s_1,u_{i+1}^n) \in \delta_1(P_1)$. This shows that $P_1$ is an $n$-ary $\delta_1$-$S_1$-primary hyperideal of $K_1$  associated to $s_1 $. In the second case, by using a similar  argument, one can easily prove that $P_2$ is an $n$-ary $\delta_2$-$S_2$-primary hyperideal of $K_2$  associated to $s_2 $.

(ii) $\Longrightarrow$ (iii) Let  $P_j$ be an $n$-ary $\delta_j$-$S_j$-primary hyperideal of $K_j$  associated to $s_j$ and $P_j=K_j$ for $j \in \{1,2\}$. Then $P_1 \times P_2$ is an $n$-ary $\hat{\delta}$-$S_1 \times S_2$-primary hyperideal of $K_1 \times K_2$ associated to $(s_1,s_2)$ by Theorem \ref{homo1}. Now, assume that $P_1$ is an $n$-ary $\delta_1$-$S_1$-primary hyperideal of $K_1$  associated to $s_1$ and $s_2 \in S_2 \cap P_2 $. Let $u_1^n \in K_1$ and $v_1^n \in K_2$ such that $g((u_1,v_1),\cdots,(u_n,v_n))=(g_1(u_1^n),g_2(v_1^n)) \in P_1 \times P_2$. Since $g_1(u_1^n) \in P_1$ and $P_1$ is an $n$-ary $\delta_1$-$S_1$-primary hyperideal of $K_1$ associated to $s_1$, we obtain  $g_1(s_1,u_i,1_{K_1}^{(n-2)}) \in P_1$ or $g_1(u_1^{i-1},s_1,u_{i+1}^n) \in \delta(P_1)$ for some $i \in \{1,\cdots,n\}$. Therefore we get $(g_1(s_1,u_i,1_{K_2}^{(n-2)}),g_2(s_2,v_i,1_{K_1}^{(n-2)})) \in P_1 \times P_2$ or $(g_1(u_1^{i-1},s_1,u_{i+1}^n),g_2(v_1^{i-1},s_2,v_{i+1}^n)) \in \delta_1(P_1) \times P_2 $. This implies that  $g_1 \times g_2((s_1,s_2),(u_i,v_i),(1_{K_1},1_{K_2})^{(n-2)}) \in P_1 \times P_2$ for some $i \in \{1,\cdots,n\}$ or  $g_1 \times g_2((s_1,s_2),(u_1,v_1),\cdots,\widehat{(u_i,v_i)},\cdots,(u_n,v_n)) \in \delta_1(P_1) \times \delta_2(P_2)=\hat{\delta}(P_1 \times P_2) $. Consequently,  $P_1 \times P_2$ is an $n$-ary $\hat{\delta}$-$S_1 \times S_2$-primary hyperideal of $K_1 \times K_2$. Similarly, one can show that $P_1 \times P_2$ is an $n$-ary $\hat{\delta}$-$S_1 \times S_2$-primary hyperideal of $K_1 \times K_2$ where $P_2$ is an $n$-ary $\delta_2$-$S_2$-primary hyperideal of $K_2$  and $s_1 \in S_1 \cap Q_1$.

(iii) $\Longrightarrow$ (i) It is obvious.
\end{proof}
\section{conclusion}
Since primary hyperideals are a crucial concept in the study of hyperstructures, extending the traditional notion of primary ideals to the context of  hyperrings, numerous researchers have investigated their extensions, yielding diverse findings. Some of the generalizations addressed in the introduction offer insight on our work. Our
purpose was to devise a broader notion encompassing $S$-primary hyperideals. To accomplish this, we constructed a hyperideal  disjoin with an $n$-ary multiplicative set and use reduction and expansion functions. This hyperideal was called $n$-ary $\phi$-$\delta$-$S$-primary. The properties of   the $n$-ary $\phi$-$\delta$-$S$-primary hyperideals  were thoroughly examined and their relations with other classes of hayperideals were explored. Furthermore, the
behavior of $n$-ary $\phi$-$\delta$-$S$-primary hyperideals is investigated on direct product of Krasner $(m,n)$-hyperrings.

As a new research subject, we suggest the notion of $\phi$-$\delta$-$S$-primary subhypermodules on Krasner $(m,n)$-hyperrins.

%%%%%%%%%%%%%%%%%%%%%%%%%%%%%%%%%%%%%%%%%%%
%%%%%%%%%%%%%%%%%%%%%%%%%

\end{document}